\documentclass[letterpaper, 12pt]{article} 
\usepackage{amsmath,amsfonts,amssymb,amsthm,xypic,xy,hyperref,mathpazo}
\xyoption{all}

\newtheorem{theorem}{Theorem}[section]
\newtheorem{corollary}[theorem]{Corollary}
\newtheorem{lemma}[theorem]{Lemma}
\newtheorem{prop}[theorem]{Proposition}

\newtheorem{remark}[theorem]{Remark}
\theoremstyle{definition}
\newtheorem{defn}[theorem]{Definition}
\newtheorem{example}[theorem]{Example}
\numberwithin{equation}{section}

\newcommand{\nc}{\newcommand}\nc{\br}{\overline}
\nc{\FH}{\mathcal H} \nc{\CC}{\mathbb C}\nc{\DD}{\mathcal D}
\nc{\Cc}{\mathcal C}\nc{\JJ}{\mathbb J}\nc{\II}{\mathcal I}
\nc{\PP}{\mathbb P} \nc{\KK}{\mathbb K} \nc{\RR}{\mathbb R}
\nc{\LL}{\mathcal L} \nc{\Ll}{\ell} \nc{\NN}{\mathbb
N} \nc{\ZZ}{\mathbb Z} \nc {\HH}{\mathbb H} \nc
{\OO}{\mathcal{O}} \nc{\lra}{\longrightarrow} \nc{\bdot}{\bullet}
\nc{\w}{\omega}
\nc{\dd}{\mathcal{D}}\nc{\im}{\mathrm{Im}}\nc{\Nij}{\mathrm{Nij}}\nc{\Jac}{\mathrm{Jac}}

\nc{\sgn}{\mathrm{sgn}}

\newcommand{\Hom}{\mathrm{Hom}}

\newcommand{\delbar}{\overline{\partial}}
\nc{\MM}{\mathcal{M}}\nc{\Har}{\mathcal{H}} \nc{\zed}{\mathcal{Z}}
\nc{\bb}[1]{\mathbb{#1}}
\nc{\image}{\mathrm{Im}\ }
\newcommand{\IP}[1]{\langle #1\rangle}

\nc{\ba}{\overline} \nc{\del}{\partial} \nc{\AAA}{\mathcal{A}}
\nc{\de}{\delta}\nc{\debar}{\overline{\delta}} \nc{\FF}{\mathcal{F}}\nc{\GG}{\mathcal{G}}\nc{\EE}{\mathcal{E}}
\nc{\Ja}{e^{\tfrac{\pi}{2}\JJ_1}} \nc{\Jb}{e^{\tfrac{\pi}{2}\JJ_2}}
\nc{\eps}{\epsilon}\nc{\id}{\mathrm{id}}\nc{\Dir}{\mathrm{Dir}}\nc{\SO}{\mathrm{SO}}
\nc{\IPS}[1]{#1}\nc{\veps}{\varepsilon}\nc{\Cour}[1]{[#1]}\nc{\Diff}{\mathrm{Diff}}
\nc{\ad}{\mathrm{ad}} \nc{\UU}{\mathcal{U}}
\nc{\Aa}{\mathcal A}\nc{\Bb}{\mathcal{B}}\nc{\Pic}{\mathrm{Pic}}\nc{\type}{\mathrm{type}}
\nc{\TT}{\mathbb{T}}\nc{\T}{\mathcal{T}}
\nc{\cl}{\mathrm{cl}}\nc{\isom}{\cong}
\nc{\Z}{\mathcal{Z}}\nc{\Tot}{\mathrm{Tot}}\nc{\Dol}{\mathrm{Dol}}\nc{\hol}{\mathrm{hol}}
\title{\bf Generalized K\"ahler geometry}
\author{Marco Gualtieri}
\date{}
\usepackage{color}
\definecolor{tocolor}{rgb}{.1,.1,.5}
\definecolor{urlcolor}{rgb}{.2,.2,.6}
\definecolor{linkcolor}{rgb}{.1,.1,.6}
\definecolor{citecolor}{rgb}{.6,.2,.1}
\hypersetup{backref=true, colorlinks=true, urlcolor=urlcolor, linkcolor=linkcolor, citecolor=citecolor}

\begin{document}

\maketitle 

\abstract{ Generalized K\"ahler geometry is the natural analogue of K\"ahler geometry, in the context of generalized complex geometry.  Just as we may require a complex structure to be compatible with a Riemannian metric in a way which gives rise to a symplectic form, we may require a generalized complex structure to be compatible with a metric so that it defines a second generalized complex structure.  We explore the fundamental aspects of this geometry, including its equivalence with the bi-Hermitian geometry on the target of a 2-dimensional sigma model with $(2,2)$ supersymmetry, as well as the relation to holomorphic Dirac geometry and the resulting derived deformation theory.  We also explore the analogy between pre-quantum line bundles and gerbes in the context of generalized K\"ahler geometry. 
}
\begingroup
\hypersetup{linkcolor=tocolor}
\tableofcontents\pagebreak
\endgroup
\section*{Introduction}

Generalized K\"ahler geometry is the natural Riemannian geometry associated to a generalized complex structure in the sense of Hitchin.  Just as in K\"ahler geometry, which involves a complex structure compatible with a symplectic form, a generalized K\"ahler structure derives from a compatible pair of generalized complex structures.  A fundamental feature of generalized geometry is that it occurs on a manifold equipped with a Courant algebroid, a structure characterized by a class in the third cohomology with real coefficients.  If this class is integral, the Courant algebroid may be thought of as arising from a rank 1 abelian gerbe.  We view this gerbe as the analogue of the prequantum line bundle in the theory of geometric quantization of symplectic manifolds. 

Just as for K\"ahler manifolds, the existence of a generalized K\"ahler structure places strong constraints on the underlying manifold; indeed, we shall see that the manifold inherits a pair of usual complex structures $(I_+,I_-)$, which need not be isomorphic as complex manifolds. Interestingly, generalized K\"ahler structures may exist on complex manifolds which admit no K\"ahler metric; indeed, if the background Courant algebroid has nonzero characteristic class, we shall see that the complex structures must fail to satisfy the $\del\delbar$-lemma and hence cannot be algebraic or even Moishezon. 

The main result of this paper is that generalized K\"ahler geometry is equivalent to a bi-Hermitian geometry discovered by Gates, Hull and Ro\v{c}ek in 1984~\cite{MR776369}, which arises on the target of a 2-dimensional sigma model upon imposing $N=(2,2)$ supersymmetry.  This equivalence, first shown in~\cite{Gualtieri:rp}, was followed by a number of results, such as those contained in~\cite{Gualtieri:ef,MR2298627, MR2217300,MR2323543,MR2249799,MR2322555,MR2371181,Goto:2007if,Gualtieri:2007fe, Goto:2009ix}, which are not easily accessed purely from the bi-Hermitian point of view. The proof presented here is more transparent than that in~\cite{Gualtieri:rp}.

The rest of the paper is concerned with understanding the relationship between the complex structures $(I_+, I_-)$ participating in the bi-Hermitian pair, from the point of view of generalized complex geometry.  We show that on each of the complex manifolds, we obtain a holomorphic Courant algebroid, which splits as a sum of transverse holomorphic Dirac structures, a situation which fits into the formalism developed by Liu, Weinstein and Xu in~\cite{MR1472888}.  Having this structure on each complex manifold, we may describe the relationship between them as a Morita equivalence between a Dirac structure on $I_+$ and its counterpart on $I_-$.  We also explain how to interpret these facts from the point of view of the prequantum gerbe, following and extending some of the work of Hull, Lindstr\"om, Ro\v{c}ek, von Unge, and Zabzine~\cite{MR2607445, MR2276362, MR2565045} on the relation between gerbes and generalized K\"ahler geometry. 

Among other things, we do not discuss the natural Bismut connections which can be used to describe and study generalized K\"ahler manifolds.  These appeared in the original work on the bi-Hermitian structure~\cite{MR776369}, and their relationship to the Courant bracket was explored in~\cite{Gualtieri:rp, MR2253158, Gualtieri:2007fe} in connection with generalized K\"ahler geometry.  We shall describe the connection approach elsewhere.  Also, we do not present the Hodge decomposition of the twisted cohomology of a generalized K\"ahler manifold, as it shall appear elsewhere~\cite{Gualtieri:ef}.

The paper is structured as follows: In Section~\ref{Courantalgebroids}, we develop the theory of holomorphic Courant algebroids, and establish the relationship between gerbes with connection and Dirac structures in Courant algebroids.  In Section~\ref{Generalizedkahlergeometry}, we define generalized K\"ahler structures, establish their equivalence with the bi-Hermitian geometry in~\cite{MR776369}, and provide several examples of generalized K\"ahler structures.  In Section~\ref{Holomorphicdiracgeometry}, we show that the generalized K\"ahler condition induces holomorphic Courant algebroids and transverse holomorphic Dirac structures on the bi-Hermitian pair, leading to a Morita equivalence relation between the constituent Lie algebroids. Finally, we explain some of the structure induced on a prequantum gerbe by the presence of a generalized K\"ahler structure.

I am grateful to Nigel Hitchin for many insights throughout the project, which began as a part of my doctoral thesis under his supervision.  I also thank Sergey Arkhipov, Henrique Bursztyn, Gil Cavalcanti, Eckhard Meinrenken and Maxim Zabzine for valuable discussions.  This research was supported by a NSERC Discovery grant.


\section{Courant algebroids}\label{Courantalgebroids}

In generalized geometry, we study geometrical structures not on the tangent bundle of a manifold but rather on the sum $\TT M:=T M\oplus T^*M$ of the tangent and cotangent bundles.  We require the geometry to be compatible, in some sense, with the orthogonal structure induced on $\TT M$ by the natural non-degenerate symmetric pairing:
\begin{equation*}
\IP{X+\xi,Y+\eta}:=\tfrac{1}{2}(\xi(Y)+\eta(X)).
\end{equation*}
Further, we may require that the geometrical structure satisfy an integrability condition, which is usually expressed using the Courant bracket on sections of $\TT M$:
\begin{equation*}
[X+\xi,Y+\eta] := [X,Y]+L_X\eta -i_Yd\xi.
\end{equation*}	
This bracket satisfies a Jacobi identity but fails to be a Lie bracket because it is not skew-symmetric. Its failure to be skew, however, is exact, and can be measured with the symmetric pairing:
\begin{equation}\label{skewfail}
[X+\xi,X+\xi] = di_X\xi = d\IP{X+\xi,X+\xi}.
\end{equation}
A striking feature of the Courant bracket is that it has symmetries fixing the underlying space $M$.  Any closed 2-form $B\in\Omega^{2,\cl}(M)$ acts on $\TT M$, preserving the Courant bracket, via the bundle map: 
\begin{equation}\label{Bfield}
e^B:X+\xi\mapsto X+\xi + i_X B.
\end{equation}
Because of this symmetry, we may ``twist'' or modify the global structure of $\TT M$ as an orthogonal bundle with a Courant bracket, keeping its local structure unchanged.  These twisted structures are therefore classified by a characteristic class in $H^1(\Omega^{2,\cl}(M))$ and are called \emph{exact Courant algebroids}~\cite{MR1472888, MR2023853, MR2171584}. 

\begin{example}\label{courconstruct}
Given a \v{C}ech cocycle $B_{ij}\in\Omega^{2,\cl}(U_{ij})$ with respect to an open cover $\{U_i\}$, we may construct a Courant algebroid in the following way.  Define a vector bundle $E$ by gluing $(\TT M)|_{U_i}$ to $(\TT M)|_{U_j}$ using the isomorphism $e^{B_{ij}}$.  The cocycle condition implies that $e^{B_{ij}}e^{B_{jk}} = e^{B_{ik}}$, so that $E$ is well-defined.  Since the $B_{ij}$ are symmetries, $E$ inherits an orthogonal structure as well as a Courant bracket.  Furthermore, from~\eqref{Bfield}, we see that the gluing maps respect the projection $\pi:\TT M\lra TM$. Hence the bundle $E$ is naturally an extension:
\begin{equation}\label{ext}
\xymatrix{0\ar[r] & T^*M\ar[r]^\iota&E\ar[r]^\pi & TM\ar[r] & 0}
\end{equation} 
\end{example}
For convenience, we provide a definition of Courant algebroid following the above line of thought.
\begin{defn}\label{Courant}
An (exact\footnote{The Courant algebroids we consider in this paper are not of the most general type considered in~\cite{MR1472888}, but are referred to as exact Courant algebroids; we will omit the adjective.}) Courant algebroid $(E,\pi,q,[\cdot,\cdot])$ is a vector bundle $E$ which is an extension of the form~\eqref{ext}, with a symmetric pairing $\IP{\cdot,\cdot} := (q(\cdot))(\cdot)$ given by a self-dual isomorphism of extensions $q:E\lra E^*$, and a bracket $[\cdot,\cdot]$ on its sheaf of sections such that, locally, there is a splitting of $\pi$ inducing an isomorphism with the orthogonal structure and Courant bracket on $\TT M$. 
\end{defn}
Also, we shall call any local splitting of $\pi$ inducing an isomorphism with the structure of $\TT M$ (as in the definition) a \emph{local trivialization} of the Courant algebroid.  Note that the local trivializations over an open set $U$ form a torsor for the group $\Omega^{2,\cl}(U)$. 

\subsection{Holomorphic Courant algebroids}

In the smooth category, the sheaf $\Omega^{2,\cl}$ of smooth closed real 2-forms has the acyclic resolution:
\begin{equation}\label{res}
\xymatrix{0\ar[r]& \Omega^{2,\cl}\ar[r] & \Omega^2\ar[r]^d & \Omega^3\ar[r]^d&\cdots},
\end{equation}
so that $H^1(\Omega^{2,\cl}(M,\RR))$ is isomorphic to $H^3(M,\RR)$.  For an explicit cocycle representing the isomorphism class of $E$, we may choose a splitting $s:TM\lra E$ of the sequence~\eqref{ext}, determining a closed 3-form measuring the failure of the splitting to be involutive: 
\begin{equation}\label{hache}
H(X,Y,Z) := \IP{s(X),[s(Y),s(Z)]}.
\end{equation}
In this way, we recover the classification of smooth exact Courant algebroids by \v{S}evera in~\cite{sevwein}.  

In our study of generalized K\"ahler geometry, we will need to understand Courant algebroids in the holomorphic category.  In this case, the sequence~\eqref{ext} may not split, and the sheaves in~\eqref{res} consist of holomorphic forms, so that the resolution is no longer acyclic.  From the exact sequence of sheaves:
\begin{equation*}
\xymatrix{0\ar[r]&\Omega^{2,\cl}\ar[r]&\Omega^2\ar[r]^\del & \Omega^{3,\cl}\ar[r]&0 },
\end{equation*}
we obtain the long exact sequence:
\begin{equation}\label{les}
\xymatrix{H^0(\Omega^2)\ar[r]^\del &H^0(\Omega^{3,\cl})\ar[r] & H^1(\Omega^{2,\cl})\ar[r] & H^1(\Omega^2)\ar[r]^\del & H^1(\Omega^{3,\cl})}.
\end{equation}
If our complex manifold satisfies the $\del\delbar$-lemma, then the left- and right-most maps in~\eqref{les} vanish and $H^0(\Omega^{3,\cl})=H^0(\Omega^3)$, exhibiting the classification of holomorphic Courant algebroids as an extension:
\begin{equation}\label{holc}
\xymatrix{0\ar[r] & H^0(\Omega^3)\ar[r] & H^1(\Omega^{2,\cl})\ar[r]& H^1(\Omega^2)\ar[r]& 0}
\end{equation}
We may interpret this as follows: given a holomorphic Courant algebroid $E$, the map to $H^1(\Omega^2)$ represents the isomorphism class of the extension~\eqref{ext}, where we observe that the extension class, which a priori lies in $H^1(\Omega^1\otimes\Omega^1)$, is forced to be skew by the orthogonal structure $q$.  On the other hand, the inclusion $H^0(\Omega^3)\hookrightarrow H^1(\Omega^{2,\cl})$ can be seen from the fact that the bracket $[\cdot,\cdot]$ of any Courant algebroid $E$ may be modified, given a holomorphic $(3,0)$-form $H\in H^0(\Omega^3)$, as follows:
\begin{equation*}
[e_1,e_2]_H := [e_1,e_2] + \iota(i_{\pi(e_1)}i_{\pi(e_2)}H).
\end{equation*}
A warning is in order, however, since we shall encounter complex manifolds which are non-K\"ahler, and hence do not necessarily satisfy the $\del\delbar$-lemma, and so the short exact sequence~\eqref{holc} may not hold.  For the general case, it will be useful to have a Dolbeault resolution of $\Omega^{2,\cl}$, as follows.  First resolve the sheaf using the local $\del\delbar$-lemma:
\begin{equation*}
\xymatrix{0\ar[r] & \Omega^{2,\cl}\ar[r] & \Z_\del^{2,0}\ar[r]^\delbar & \Z_\del^{2,1}\ar[r]^\delbar &\cdots},
\end{equation*} 
where $\Z_\del^{p,q}$ is the sheaf of smooth $\del$-closed $(p,q)$ forms. Then use the Dolbeault resolution for each $\Z_\del^{p,q}$ given by $\del$ to conclude that $H^1(\Omega^{2,\cl})$ is given by the first total cohomology of the double complex
\begin{equation}\label{resolu}
\xymatrix{
\Omega^{4,0} & &\\
\Omega^{3,0}\ar[u]^{\del}\ar[r]^\delbar & \Omega^{3,1} & \\
\Omega^{2,0}\ar[u]^{\del}\ar[r]^\delbar & \Omega^{2,1}\ar[u]^{\del}\ar[r]^\delbar &\Omega^{2,2} 
}
\end{equation}
\begin{prop}\label{classhol}
Holomorphic Courant algebroids are classified up to isomorphism by
\begin{equation}\label{holcour}
H^1(\Omega^{2,\cl})=\frac{(\Omega^{3,0}(M)\oplus \Omega^{2,1}(M))\cap\ker d}{d(\Omega^{2,0}(M))}.
\end{equation}
\end{prop}
\begin{example}\label{genex}
Given a Dolbeault representative $T+H$ for a class in~\eqref{holcour}, with $T\in\Omega^{2,1}$ and $H\in\Omega^{3,0}$, we may construct a corresponding holomorphic Courant algebroid.  Viewing $T$ as a map $T:T_{1,0}M\lra T^*_{0,1}M\otimes T^*_{1,0}M$, define the following partial connection on sections of $E=T_{1,0}M\oplus T^*_{1,0}M$:
\begin{equation*}
\overline D = \begin{pmatrix}\delbar& 0\\T&\delbar\end{pmatrix}:\Gamma^\infty(E)\lra \Gamma^\infty(T^*_{0,1}M\otimes E),
\end{equation*}
where $\delbar$ are the usual holomorphic structures on the tangent and cotangent bundles. $\overline D$ squares to zero and defines a new holomorphic structure on the complex bundle $E$, which then becomes a holomorphic extension of $T_{1,0}$ by $T^*_{1,0}$. The symmetric pairing on $E$ is the usual one obtained from the duality pairing, but the bracket is twisted as follows:
\begin{equation}\label{twcour}
[X+\xi,Y+\eta] = [X,Y] + L_X\eta - i_Y\xi + i_Xi_YH.
\end{equation}
Under the assumption that $d(T+H)=0$, this bracket is well-defined on the sheaf of $\overline D$-holomorphic sections of $E$, and defines a holomorphic Courant algebroid, as required. 
\end{example}

\begin{example}\label{hopf}
Consider the standard Hopf surface $X = (\CC^2\backslash\{0\})/\ZZ$, where $\ZZ$ acts by $\varphi:(x^1,x^2)\mapsto (2x^1,2x^2)$.  This is a complex manifold diffeomorphic to $S^1\times S^3$, and hence it does not admit a K\"ahler structure.  The Hodge numbers $h^{p,q}$ all vanish except $h^{0,0}=1$ , $h^{0,1}=1$, $h^{2,1}=1$ and $h^{2,2}=1$.  Hence, the group~\eqref{holcour} classifying holomorphic Courant algebroids is $H^1(\Omega^2)\isom\CC$.  The Courant algebroids in this 1-parameter family may be described explicitly using Example~\ref{genex}, but also holomorphically, as follows. The union of the two elliptic curves $E_1 = \{x_1=0\}$ and $E_2=\{x_2=0\}$ form an anticanonical divisor, corresponding to the meromorphic section $B = c(x_1x_2)^{-1}dx_1\wedge dx_2,\ c\in\CC$. Glue $\TT(X\backslash E_1)$ to $\TT(X\backslash E_2)$ using the holomorphic closed 2-form $B$ to obtain a Courant algebroid with modified extension class.
\end{example}

On any complex manifold, there is an injection of sheaves from the holomorphic closed 2-forms to the smooth real closed 2-forms, which we denote, for disambiguation,
\begin{equation*}
\Omega^{2,\cl}_{\hol}\lra \Omega^{2,\cl}_{\infty}(\RR).
\end{equation*}
This morphism is given by $B^{2,0}\mapsto \tfrac{1}{2}(B^{2,0}+\overline{B^{2,0}})$. For this reason we have a map $H^1(\Omega^{2,\cl}_{hol})\lra H^1(\Omega^{2,\cl}_\infty)$; indeed, the underlying real vector bundle of any holomorphic Courant algebroid is itself a smooth real Courant algebroid.  
\begin{example}\label{hopf2}
The real Courant algebroid on $X=S^3\times S^1$ corresponding to the holomorphic Courant algebroid described in Example~\ref{hopf} is easily obtained by choosing a Dolbeault representative of the \v{C}ech cocycle $\tfrac{1}{2}(B+\overline{B})$. For $B=c(x_1x_2)^{-1}dx_1\wedge dx_2,\ c\in\CC$, we obtain a class in $H^3(X,\RR)$ which evaluates on the fundamental cycle of $S^3$ to $-4\pi^2\mathrm{Re}(c)$.
\end{example}

\subsection{Gerbes with connection}

In this section we will describe the relationship between $\CC^*$-gerbes and Courant algebroids, and in particular what the meaning of a Dirac structure is from the point of view of a gerbe.  This is useful for understanding pre-quantization conditions in generalized K\"ahler geometry but is not necessary for understanding the geometry per se.  We will take the \v{C}ech approach of~\cite{MR1197353, MR1876068, Chatterjee} to the description of gerbes, omitting  discussions of refinements of covers, for convenience. We treat gerbes in both the smooth and holomorphic categories, indicating differences as we go.  

First, we review the basic method of working with gerbes at the \v{C}ech level. Let $M$ be a smooth real or complex manifold, where $\OO_M$ denotes the sheaf of complex-valued functions (smooth or holomorphic, respectively).  Choose an open covering $\{U_i\}$, and let $G$ be a $\CC^*$-gerbe which is locally trivialized over this covering, so that it is given by the data $\{L_{ij}, \theta_{ijk}\}$, where $L_{ij}$ are (smooth or holomorphic) complex line bundles over $U_{ij}$, chosen so that $L_{ij}$ is dual to $L_{ji}$, and  $\theta_{ijk}:L_{ij}\otimes L_{jk}\lra L_{ik}$ are isomorphisms of line bundles over $U_{ijk}$ such that on quadruple overlaps $U_{ijkl}$ we have the coherence condition:  
\begin{equation*}\label{gluger}
\theta_{ikl}\circ(\theta_{ijk}\otimes \id) = \theta_{ijl}\circ(\id\otimes\theta_{jkl}).
\end{equation*}
Two local trivializations of the same gerbe differ by a collection of line bundles $\{L_i\rightarrow U_i\}$, which affect the above data via:
\begin{equation*}
\{L_i\}:(L_{ij},\theta_{ijk})\longmapsto(L_{ij}\otimes L_i\otimes L_j^{*},\ \ \ \theta_{ijk}\otimes\id_{L_i\otimes L_{k}^{*}}).
\end{equation*}
A global trivialization, or \emph{object}, of $G$ may be described, with respect to the local trivialization above, by line bundles $\{L_i\}$ together with isomorphisms $m_{ij}:L_i\lra L_{ij}\otimes L_j$, such that on triple overlaps we have: 
\begin{equation*}
m_{ik}=(\theta_{ijk}\otimes\id_{L_k})\circ (\id_{L_{ij}}\otimes m_{jk})\circ m_{ij}.
\end{equation*}
Two global objects $S=\{L_i,m_{ij}\}, S'=\{L'_i,m'_{ij}\}$ of a gerbe differ by a global line bundle $L_{SS'}$ defined by $L_{SS'}|_{U_i}:= L^*_i\otimes L'_i$. As a result we may define a category of objects, where 
\begin{equation*}
\Hom(S,S')= \Gamma(M,L_{SS'}). 
\end{equation*}
Finally, an \emph{equivalence} $G\rightarrow G'$ of gerbes is a global trivialization of $G^*\otimes G'$, where duality and  tensor product of gerbes is defined in the obvious way. An auto-equivalence is a global object in the trivial gerbe: hence it is a global line bundle.
$\CC^*$-gerbes are classified up to equivalence by 
$H^2(\OO^*_M)$, where $\OO^*_M\subset \OO_M$ is the subsheaf of nowhere-vanishing functions.  

We now introduce the notion of a gerbe connection over an arbitrary Lie algebroid $A$.  Let $(A,a,[\cdot,\cdot])$ be a complex Lie algebroid on $M$, where $[\cdot,\cdot]$ is the Lie bracket on the sheaf of sections of $A$, and $a:A\lra\mathrm{Der}(\OO_M)$ is the bracket-preserving bundle map to the tangent bundle, usually called the anchor.  We will use the notation $(\Omega^\bullet_A, d_A)$ to denote the associated de Rham complex of $A$.  Note that if $A$ is a holomorphic Lie algebroid, the anchor maps to the holomorphic tangent bundle, whereas in the smooth case it maps to $TM\otimes\CC$. 

\begin{defn}\label{Acon} An $A$-connection on a line bundle $L$ is a differential operator 
\begin{equation*}
\del:\OO(L)\lra \OO(A^*\otimes L),
\end{equation*}
such that $\del(fs) = (d_Af)\otimes s +f\del s$, for $f\in\OO_M$ and $s\in\OO(L)$. As with usual connections, $\del$ has a curvature tensor $\del^2=F_\del\in\Omega^2_A(M)$ such that $d_AF_\del = 0$. When $F_\del= 0$, we say that $L$ is flat over $A$, or that $L$ is an $A$-module.  
\end{defn}

\begin{defn}\label{gacon} An $A$-connection $(\del,B)$ on the gerbe $G$ defined by $\{L_{ij},\theta_{ijk}\}$ is given as follows.  The first component, $\del$, called the $0$-connection, is a family of $A$-connections $\del_{ij}$ on $L_{ij}$ with $\del_{ji} = \del_{ij}^*$ and such that $\theta_{ijk}$ is flat in the induced connection:
\begin{equation*}
\theta_{ijk}\circ (\del_{ij}\otimes 1 + 1\otimes \del_{jk}) = \del_{ik}\circ\theta_{ijk}.
\end{equation*}
The second component, $B$, called the $1$-connection, is a collection $\{B_i\in \Omega^2_A(U_i)\}$ satisfying $B_j - B_i= F_{\del_{ij}}$. The global 3-form $H\in\Omega^3_A(M)$ defined by $H|_{U_i} = d_AB_i$ is called the curving of the connection, and satisfies $d_A H = 0$. When $H=0$, we say that $G$ is flat over $A$.  
\end{defn}
Let $S:G\rightarrow G'$ be an equivalence of gerbes, and choose trivializations as above so that the object $S$ in $G^*\otimes G'$ is given by $\{L_i,m_{ij}\}$, where $m_{ij}$ are isomorphisms
\begin{equation*}
m_{ij}:L_i\lra (L^*_{ij}\otimes L'_{ij})\otimes L_j.
\end{equation*}
Also, let $G,G'$ be equipped with $A$-connections $(\del,B),(\del',B')$ as above.  Then, by definition, to promote $S$ to an equivalence of gerbes with connection is to equip $L_i$ with $A$-connections $\del_i$ such that 
\begin{equation*}
{m_{ij}}_*(\del_i) = \del_{ij}^* + \del'_{ij} + \del_j\ \ \ \ \text{and}\ \ \ \ B_i - B'_i = F_{\del_i}.
\end{equation*}   
An auto-equivalence with connection is then simply a line bundle with connection $(L,\del)$; its action is only seen by the data defining the 1-connection, via $B_i\mapsto B_i - F_\del|_{U_i}$.
Gerbes with $A$-connections are classified up to equivalence by the hypercohomology group
\begin{equation*}
\HH^2(\xymatrix{\OO^*_M\ar[r]^{d_A\log} &\Omega^1_A\ar[r]^{d_A}& \Omega^2_A}).
\end{equation*}

As in the case of holomorphic vector bundles, where the existence of a holomorphic connection is obstructed by the Atiyah class, the existence of an $A$-connection on a gerbe is obstructed in general.  We now briefly summarize the treatment of the obstructions given in~\cite{Chatterjee}. 

Arbitrarily choose $A$-connections $\del_{ij}$ on $L_{ij}$, so that $\theta_{ijk}\in \OO(L_{ik}\otimes L_{kj}\otimes L_{ji})$ is not necessarily flat for the induced connection $\del_{ikj}$:
\begin{equation*}
\del_{ikj}\theta_{ijk} = A_{ijk}\otimes\theta_{ijk}.
\end{equation*} 
This defines a \v{C}ech cocycle $\{A_{ijk}\in \Omega^1_A(U_{ijk})\}$, which represents the 0-Atiyah class $\alpha_0=[A_{ijk}]\in H^2(\Omega^1_A)$ obstructing the existence of a $0$-connection on the gerbe. 
If $\alpha_0=0$, then there exists a $0$-connection $\{\del_{ij}\}$, and the curvatures $\{F_{\del_{ij}}\in\Omega^2_A(U_{ij})\}$ define the $1$-Atiyah class $\alpha_1\in H^1(\Omega^2_A)$ obstructing the existence of a $1$-connection.   

The following is an example of how $A$-connections on gerbes may be used. It is an analog of the well-known result for complex vector bundles that a flat partial $(0,1)$-connection induces a holomorphic structure on the bundle. It is essentially a realization of the following isomorphism:
\begin{equation*}
H^2(\OO^*_{\hol}) \isom \HH^2(\xymatrix{\OO^*_{\infty}\ar[r]^{\delbar\log} & \Omega^{0,1}\ar[r]^\delbar & \Omega^{0,2}\ar[r]^\delbar & \Omega^{0,3}}).
\end{equation*}
\begin{theorem}\label{cangerb}
Let $G$ be a smooth $\CC^*$-gerbe over a complex manifold $M$, and let $A=T_{0,1}M$ be the Dolbeault Lie algebroid.  The choice of a flat $A$-connection on $G$ naturally endows the gerbe with a holomorphic structure.
\end{theorem}
\begin{proof}
Choose a smooth trivialization $\{L_{ij},\theta_{ijk}\}$ for the gerbe, and let $\{{D}_{ij}, B_i\}$ be the given flat $A$-connection.  Since $\delbar B_i=0$, we make a choice of potential $A=\{A_i\in\Omega^{0,1}(U_i)\}$ such that $\delbar A_i = B_i$ (refining the cover as necessary).  Then the trivial auto-equivalence of $G$, equipped with the $A$-connection $\delbar + A_i$, defines an equivalence of gerbes with connection:
\begin{equation*}
(L_{ij},\theta_{ijk}, {D}_{ij}, B_i)\lra (L_{ij},\theta_{ijk}, D_{ij}+A_i-A_j,0).
\end{equation*}
The resulting $A$-connections $D_{ij}+A_i-A_j$ are therefore flat, rendering both $L_{ij}$ and $\theta_{ijk}$ holomorphic, and therefore defining a holomorphic gerbe $\GG_A$, which a priori depends on $A$.  A different choice $A'=\{A_i':\delbar A_i'=B_i\}$, similarly, defines a holomorphic gerbe $\GG_{A'}$. But, $\delbar(A_i'-A_i)=0$, so that we obtain an equivalence of holomorphic gerbes
\begin{equation}\label{difpot}
S_{AA'}:\GG_A\lra \GG_{A'},
\end{equation}
using the line bundles $L_i:=U_i\times\CC$ equipped with the holomorphic structures $\delbar + A_i'-A_i$.  This shows that the holomorphic gerbes $\GG_A,\GG_{A'}$ are canonically equivalent. Furthermore, for three choices $A, A', A''$ of potentials for $\{B_i\}$, there is a canonical isomorphism commuting the diagram 
\begin{equation}
\xymatrix@R=1.4em@C=.2em{&\GG_{A'}\ar[dr]^{S_{A'A''}} & \\\GG_A\ar[ru]^{S_{AA'}}\ar[rr]_{S_{AA''}} &  & \GG_{A''}}
\end{equation}
in the sense that there is a canonical isomorphism of functors $S_{A'A''}\circ S_{AA'}\to S_{AA''}$.  This isomorphism necessarily satisfies the tetrahedral coherence condition for a quadruple of   potentials.  Thus we have defined a holomorphic gerbe $\GG$ independent of the choice of potential. Note that this holomorphic gerbe does not come with a preferred holomorphic local trivialization.

We now show that $\GG$ is independent of the original local trivialization $\{L_{ij}, \theta_{ijk}, {D}_{ij}, B_i\}$ of the gerbe with connection.  Changing the trivialization using local line bundles with connection $\{L_i,{D}_i\}$, we obtain new trivialization data $\{\tilde L_{ij}, \tilde\theta_{ijk}, \tilde{D}_{ij}, \tilde B_i\}$, given by:
\begin{equation}\label{chngerbcon}
(L_{ij}\otimes L_i\otimes L_j^*,\ \ \ \theta_{ijk}\otimes\id_{L_i\otimes L_{k}^*},\ \ \ {D}_{ij}+{D}_i + {D}_j^*,\ \ \  B_i - F_{{D}_i}).
\end{equation}
Choosing $A=\{A_i:\delbar A_i = B_i\}$ and $\tilde A=\{\tilde A_i:\delbar \tilde A_i = \tilde B_i\}$, we obtain holomorphic gerbes $\GG_A,\GG_{\tilde A}$ by the procedure above.  But the choice of $A,\tilde A$ immediately implies that $F_{D_i} = \delbar(\tilde A_i - A_i)$, and therefore that $D_i + \tilde A_i -  A_i$ defines a holomorphic structure on $L_i$, so that $(L_i,D_i + \tilde A_i - A_i)$ defines a holomorphic equivalence $S_{A\tilde A}:\GG_A\to\GG_{\tilde A}$, as required. Finally, we observe that there is a canonical isomorphism commuting the diagram:
\begin{equation*}
\xymatrix@R=1.2em{\GG_{A}\ar[d]_{S_{AA'}}\ar[r]^{S_{A\tilde A}} & \GG_{\tilde A}\ar[d]^{S_{\tilde A \tilde A'}}\\ \GG_{A'}\ar[r]_{S_{A'\tilde A'}} & \GG_{\tilde A'}}
\end{equation*}
Combining this isomorphism with that from~\eqref{difpot}, we obtain canonical isomorphisms commuting each face of the diagram relating each of the gerbes obtained from different choices ($A, A', A''$ are potentials for $\{B_i\}$, while $\tilde A,\tilde A',\tilde A''$ are potentials for $\{\tilde B_i\}$):
\begin{equation*}
\xymatrix@C=1em@R=.5em{
\GG_{A}\ar[dd]\ar[dr]\ar[rr]  &  &  \GG_{\tilde A}\ar'[d][dd]\ar[dr]  & \\
   &\GG_{A'}\ar[dl]\ar[rr]&           &\GG_{\tilde A'}\ar[dl]\\
\GG_{A''}\ar[rr]&  & \GG_{\tilde A''}& \\ }
\end{equation*}
The composition of all the isomorphisms commuting the faces of the above diagram yields the identity map on any of the edges.  This establishes that the holomorphic gerbe $\GG$ is independent of the initial local trivialization.
\end{proof}

Our purpose in introducing connections on gerbes is twofold. First, taking $A$ to be the tangent bundle\footnote{We could consider any Lie algebroid $A$, obtaining Courant algebroids which are extensions of $A$ by $A^*$.}, we associate, following~\cite{MR2013140}, a canonical Courant algebroid $E_\del$ to a $0$-connection $\del$ on a gerbe.  Second, we show that any Dirac structure $D\subset E_\del$ induces a flat $D$-connection on the gerbe.     
\begin{theorem}\label{corresp}
To every $0$-connection $\del$ (over the complexified or holomorphic tangent bundle) on a gerbe $G$, there is a canonically associated Courant algebroid $E_\del$, with isomorphism class given by the map
\begin{equation}\label{gerbcour}
\HH^2(\xymatrix{\OO^*_M\ar[r]^-{d\log} &\Omega^1})\stackrel{d}{\lra} \HH^2(\xymatrix{0\ar[r]&\Omega^2\ar[r]^{d}&\Omega^3\ar[r]^d&\Omega^4}) = H^1(\Omega^{2,\cl}).
\end{equation}
Furthermore, global splittings of $E_\del\stackrel{\pi}{\lra}TM$ with isotropic image correspond bijectively with $1$-connections on $(G,\del)$. A connection is flat if and only if the corresponding splitting is involutive. 
\end{theorem}
\begin{proof}
Let the $0$-connection be given by $a=\{\del_{ij}\}$ in a local trivialization of the gerbe, and define the Courant algebroid $E_a$ by the procedure in Example~\ref{courconstruct}, gluing $\TT{U_i}$ to $\TT{U_j}$ using $F_{\del_{ij}}$.  We must check that the Courant algebroid is independent of $a$.  Change the local trivialization, using local line bundles with connection $g=\{L_i,D_i\}$, so that the $0$-connection is given by the collection $a^g=\{\del_{ij} + D_i + D_j^*\}$ of connections on $L_{ij}\otimes L_i\otimes L_j^*$.  Hence $E_{a^g}$ is constructed using $F_{\del_{ij}} + F_{D_i} - F_{D_j}$.  But then we obtain a map $\psi_g:E_a\lra E_{a^g}$ defined by $\psi_g|_{U_i} = e^{F_{D_i}}$, which is an isomorphism of Courant algebroids because it intertwines the gluing maps with isomorphisms of the Courant structure, i.e. the following diagram commutes:
\begin{equation*}
\xymatrix@C=4em{
\TT U_i\ar[d]_{e^{F_{D_i}}}\ar[r]^{e^{F_{\del_{ij}}}} & \TT U_j\ar[d]^{e^{F_{D_j}}}\\
\TT U_i\ar[r]_{e^{F_{\del_{ij}}+ F_{D_i} - F_{D_j}}} & \TT U_j
}
\end{equation*}
Functoriality follows from the fact that if $g = g_1g_2$ is the tensor product of local line bundles with connection then $\psi_g = \psi_{g_1}\psi_{g_2}$.  Hence we obtain a well-defined Courant algebroid $E_\del$ associated to the $0$-connection.

Trivializing the line bundles $L_{ij}$ so that the gerbe is given by data $g_{ijk}\in\OO^*(U_{ijk})$ and $\del_{ij}$ is given by connection $1$-forms $A_{ij}$, we see that the gluing $2$-forms for the Courant algebroid are simply $dA_{ij}$; hence $[E_\del]$ is indeed given by $d[(G,\del)]$ as in~\eqref{gerbcour}. 

A $1$-connection $B$ on $(G,\del)$ consists of 2-forms $B_i\in\Omega^2(U_i)$ with $B_j-B_i=F_{\del_{ij}}$, data which determines a splitting $s_{B}$ of $E_a\stackrel{\pi}{\lra} TM$ defined by $s_B|_{TU_i}(X) = X + i_X B_i\in\TT U_i$: clearly $e^{F_{\del_{ij}}}\circ s_B|_{TU_i} = s_B|_{TU_j}$ on $U_{ij}$, rendering $s_B$ well-defined.  The map $B\mapsto s_B$ is clearly a bijection, since isotropic splittings of $\TT U_i\lra TU_i$ are simply graphs of $2$-forms.  The bijection is functorial because a change of local trivialization $g=\{L_{i},D_i\}$  maps $B=\{B_i\}$ to $B^g = \{B_i - F_{D_i}\}$, so that the following diagram commutes:
\begin{equation*}
\xymatrix{
E_a\ar[d]_{\psi_g} & TM\ar[l]_{s_B}\ar[d]^{\id}\\
E_{a^g} & TM\ar[l]_{s_{B^g}}
}
\end{equation*}
Finally, the graph of a 2-form in $\TT M$ is involutive if and only if it is closed, hence the splitting $s_B$ is involutive if and only if the gerbe is flat over $TM$. 
\end{proof}
\begin{remark}\label{gauge}
A self-equivalence (i.e. gauge transformation) of a gerbe with $0$-connection is given by tensoring by a global line bundle with connection $(L,D)$.  This does not affect the gerbe with $0$-connection, but according to the above result, $\psi_g$ need not be the identity map: indeed, in this case $F_{D}$ defines a global closed (integral) $2$-form, so that $\psi_g$ is an automorphism of $E_\del$ of the form $\psi_g = e^{F_D}$. 
\end{remark}

Proposition~\ref{gerbcour} is stated for the smooth complexified tangent bundle or the holomorphic tangent bundle.  To obtain smooth real Courant algebroids, we equip the gerbe with a Hermitian structure~\cite{MR1876068} and require the $0$-connection $\del$ to be unitary.  We explain this in detail below.
\begin{defn}
A Hermitian structure on the gerbe $G$ defined by $\{L_{ij},\theta_{ijk}\}$ is given by a family of Hermitian metrics $h=\{h_{ij}\}$ on the complex line bundles $L_{ij}$, such that $\theta_{ijk}$ is unitary.   A connection on $G$ defined by $\{\nabla_{ij},B_i\}$ is unitary when $\nabla_{ij}$ are unitary connections (with \emph{real} curvatures, by convention) and $B_i$ are real.
\end{defn}
Two local trivializations of a Hermitian gerbe with unitary connection differ by local Hermitian line bundles $\{L_i, h_i, \nabla_i\}$ with unitary connection; this acts on the gluing data as in~\eqref{chngerbcon}, with the local Hermitian metrics $h_{ij}$ mapped to $h_{ij}h_ih_j^{-1}$.
 
\begin{corollary}\label{unitgerb}
To every unitary $0$-connection $\nabla = \{\nabla_{ij}\}$ (over the real tangent bundle) on a Hermitian gerbe $(G,h)$, there is a canonically associated real Courant algebroid $E_\nabla$, with isomorphism class given by the map
\begin{equation}\label{surjonint}
\HH^2(\xymatrix{\OO(U(1))\ar[r]^-{-id\log} &\Omega^1_\RR})\stackrel{d}{\lra} H^1(\Omega^{2,\cl}_\RR) = H^3(M,\RR).
\end{equation}
The correspondence between splittings and $1$-connections is as before.
\end{corollary}
\begin{example}\label{gcbab}
Let $(G,h)$ be a holomorphic Hermitian gerbe.  Then on each Hermitian holomorphic line bundle $L_{ij}$, we have the unique Chern connection $\nabla_{ij}$.  This defines a canonical unitary $0$-connection $\nabla$ on $G$.  By the corollary, we obtain a canonical real Courant algebroid $E_\nabla$. 

Of note in this example is the fact that the gluing maps $e^{F_{\nabla_{ij}}}:\TT U_i\lra \TT U_j$ defining $E_\nabla$ intertwine the automorphisms 
\begin{equation*}
\JJ_i := \left.\begin{pmatrix}I & 0\\ 0 & -I^*\end{pmatrix}\right|_{U_i}:\TT U_i\lra \TT U_j,
\end{equation*} 
where $I:TM\lra TM$ is the complex structure on the underlying manifold.  In fact, the global operator $\JJ:E_\nabla\lra E_\nabla$, defined by $\JJ|_{U_i} = \JJ_i$, is a simple example of a generalized complex structure~\cite{Gualtieri:qr}; in this case it arises directly as a consequence of choosing a Hermitian structure on a holomorphic gerbe.  

Having the generalized complex structure $\JJ$, a unitary $1$-connection for $(G,h,\nabla)$ compatible with the holomorphic structure may be described as a section $s$ of $\pi:E_\nabla\lra TM$  which is complex-linear, i.e. $\JJ\circ s = s\circ I$. As discussed in~\cite{Chatterjee}, there is not a distinguished unitary $1$-connection. 
\end{example}

\subsection{Gerbes and Dirac structures}\label{gerbdir}

The Courant bracket was initially introduced~\cite{CourWein} because, in a sense, it provides a unified source for many interesting Lie algebroids.  Indeed, because the failure~\eqref{skewfail} of the bracket to be Lie is measured by the symmetric pairing, it follows that any subbundle $D\subset E$ of a Courant algebroid which is isotropic and involutive inherits a Lie algebroid structure, by simply restricting the bracket and projection $\pi$ to the subbundle $D$.  The examples which inspired Courant and Weinstein were those of a bivector field $\pi\in\Gamma(\wedge^2 TM)$ and a 2-form $\omega\in\Omega^2(M)$.  Viewing these tensors as maps $TM\lra T^*M$ and $T^*M\lra TM$ respectively, their graphs $\Gamma_\pi\subset \TT M$, $\Gamma_\omega\subset\TT M$ are always isotropic; they are involutive if and only if $\pi$ is Poisson and $\omega$ is closed.  An involutive isotropic subbundle is called a Dirac structure when it is maximally isotropic, but we shall encounter other, non-maximal examples.  The remainder of this section is concerned with the question of how a gerbe with $0$-connection $(G,\del)$, giving rise to a Courant algebroid $E_\del$, is affected by the presence of an involutive isotropic subbundle $D\subset E_\del$.  

\begin{theorem}\label{diracflat}
Let $E_\del$ be the Courant algebroid associated to a gerbe with $0$-connection $(G,\del)$.  Given an involutive isotropic subbundle $D\subset E_\del$, the gerbe $G$ obtains a canonical flat $D$-connection.
\end{theorem}
\begin{proof}
The restriction of $\pi:E_\del\lra TM$ as well as the Courant bracket to $D$ gives it the structure of a Lie algebroid, and by choosing a local trivialization for $(G,\del)$, we immediately obtain a $0$-connection $\del^D$ over $D$ by composition: 
\begin{equation*}
\del^D_{ij}:=\pi|_D^*\circ\del_{ij}.     
\end{equation*}
To obtain the $1$-connection $B$ over $D$, write the inclusion $D\subset E_\del$ locally, as involutive isotropic subbundles $D_i\subset \TT U_i$ such that $e^{F_{\del_{ij}}}D_i = D_j$.  Then consider the antisymmetric pairing on $\TT U_i$:
\begin{equation*}
\IP{X+\xi,Y+\eta}_- := \tfrac{1}{2}(\xi(Y)-\eta(X)).
\end{equation*}    
This restricts to $D_i$ and determines 2-forms $B_i\in\Omega^2_D(U_i)$. The gluing condition $e^{F_{\del_{ij}}}D_i = D_j$ implies that $B_i-B_j = \pi|_D^*F_{ij} = F_{\del^D_{ij}}$, 
so that $\{\del^D_{ij},B_i\}$ is indeed a $D$-connection. 

We now check that the $D$-connection is independent of the local trivialization used to define it.  In a local trivialization differing from the initial one by the local line bundles with connection $g=\{L_{i},\del_{i}\}$, the $0$-connection over $D$ is given by $\pi|_D^*(\del_{ij} + \del_i +\del_j^*)$, and the effect on $E_\del$ is via the isomorphism $\psi_g$, which sends $D_i$ to $e^{F_{\del_i}}D_i$, so that the restriction of the antisymmetric pairing to $D_i$ yields $B_i + \pi|_D^*F_{\del_i}$.  The resulting expression for the $D$-connection is precisely that obtained by changing the local trivialization of $(G,\del^D)$ by
the local line bundles with $D$-connection $g^D:=\{L_i, \pi|_D^*\del_i\}$.  The naturality of the map $g\mapsto g^D$ ensures that $(G, \del^D, B)$ is canonically defined.

Finally, the curving of the $D$-connection may be computed using a  general property of the Courant bracket implicit in Theorem 2.3.6. of~\cite{MR998124}, namely that the restriction of $\IP{\cdot,\cdot}_-$ to any isotropic integrable subbundle $D_i\subset \TT M$ is closed with respect to the algebroid differential. 
\end{proof}

We now show that the above theorem may be applied in order to endow a gerbe with a holomorphic structure, in such a way that the resulting holomorphic gerbe inherits a holomorphic $0$-connection.  At the level of Courant algebroids, this becomes a reduction procedure as in~\cite{MR2323543}, whereby a real Courant algebroid ``reduces'' to a holomorphic one. 

Let $(G,\del)$ be a gerbe with $0$-connection, given by $\{L_{ij}, \del_{ij}\}$, over a complex manifold $X$, and let $E_\del$ be the associated complex Courant algebroid.  Note that $G$ is not assumed to have a holomorphic structure; our first goal will be to endow $G$ with such a structure.  We will do this by choosing an involutive isotropic subbundle $D\subset E_\del$ such that $\pi|_D : D\lra T_{0,1}$ is an isomorphism.  In other words, $D$ is a lifting of the anti-holomorphic tangent Lie algebroid $T_{0,1}$ to an involutive isotropic subbundle of $E_\del$. This choice induces a holomorphic structure on $G$, by Theorem~\ref{diracflat}. 
\begin{defn}\label{deflift}
Let $E$ be a complex Courant algebroid over a complex manifold $X$.  A \emph{lifting} of $T_{0,1}X$ to $E$ is a isotropic, involutive subbundle $D\subset E$ mapping isomorphically to $T_{0,1}X$ under $\pi:E\lra TX\otimes\CC$.
\end{defn}
The existence of a lifting for $T_{0,1}X$ as above will be controlled by an obstruction map which we now describe.  Consider the short exact sequence of vertical complexes:
\begin{equation*}
\xymatrix{
& &\Omega^{1}_X\ar[r]^-\del& \Omega^{2,\cl}_X\ar[r] &0 & \\
0\ar[r]&\CC\ar[r]&\OO_X\ar[u]^-\del& &
}
\end{equation*}
This gives the following excerpt from the long exact sequence:
\begin{equation}\label{realhol}
\xymatrix{H^1(\Omega^{2,\cl}_X)\ar[r]^\epsilon & H^3(X,\CC)\ar[r]^-\gamma & \HH^3(\OO_X\stackrel{\del}{\rightarrow}\Omega^1_X)}
\end{equation}
\begin{lemma}\label{obs}
Let $E$ be a complex Courant algebroid over the complex manifold $X$.  There exists a lifting of $T_{0,1}$ to $E$ if and only if $\gamma([E])=0$ in $\HH^3(\OO_X\stackrel{\del}{\rightarrow}\Omega^1_X)$. 
\end{lemma}
\begin{proof}
Choose an isotropic splitting $s:TX\lra E$, which determines a $3$-form $H$ as in Equation~\eqref{hache}, so that $E$ is isomorphic as a Courant algebroid to $\TT X\otimes\CC$ equipped with the Courant bracket twisted by $H$ as in Equation~\eqref{twcour}; indeed $[E] = [H]\in H^3(M,\CC)$.  Then a general isotropic lifting of $T^{0,1}X$ is given by
\begin{equation*}
D=\{X + i_X\theta\ :\ X\in T^{0,1}X,\ \ \theta\in \Omega^{1,1}(X)\oplus \Omega^{0,2}(X)\},
\end{equation*} 
and $D$ is involutive if and only if $(d\theta - H)^{(1,2)+(0,3)}=0$, or in other words 
\begin{equation}\label{potent}
H^{1,2}+H^{0,3} = -\delbar \theta^{1,1} - d\theta^{0,2}.
\end{equation} 
Using the Dolbeault resolution of $\OO_X\stackrel{\del}{\rightarrow}\Omega^1_X$, we conclude that $\theta$ exists if and only if $\gamma([H])=0$.
\end{proof}
\begin{remark}\label{explclass}
A solution to equation~\eqref{potent} defines a cocycle $H^{3,0} + H^{2,1} + \del\theta^{1,1}\in \Z^1(\Omega^{2,\cl}_X)$ (using the resolution~\eqref{resolu}), since 
\begin{equation*}
d(H^{3,0} + H^{2,1} + \del\theta^{1,1}) = \delbar H^{2,1} -\del\delbar\theta^{1,1}=\delbar H^{2,1} + \del(H^{1,2} + \del \theta^{0,2}) = 0.
\end{equation*} 
Furthermore, we may change the isotropic splitting $s$ in the proof above by a global smooth $2$-form $B$, which sends $H\mapsto H-dB$ and modifies the lifting via $\theta\mapsto \theta + B^{1,1} + B^{0,2}$.  As a result, the cocycle condition~\eqref{potent} holds independently of the choices made.  In this way, a lifting of $T_{0,1}X$ to $E$ naturally induces a  lifting of $[E]\in H^3(M,\CC)$ to $H^1(\Omega^{2,\cl}_X)$ in the exact sequence~\eqref{realhol}.  
\end{remark}
\begin{remark}
The map $\epsilon$ in~\eqref{realhol} is lifted to a natural operation on Courant algebroids in~\cite{Grutzmann:2010sp}, where it is shown that any holomorphic Courant algebroid $\EE$ induces a smooth complex Courant algebroid structure on $\EE\oplus (T_{0,1}M\oplus T^*_{0,1}M)$, called the companion matched pair of $\EE$. 
\end{remark}

We see from~\eqref{realhol} that if $\gamma([H])=0$, then $[H]$ is in the image of a map from $H^1(\Omega^{2,\cl}_X)$, which as we know from Proposition~\ref{classhol} is the space classifying holomorphic Courant algebroids.  Indeed, Lemma~\ref{obs} and the above remark imply that a complex Courant algebroid with a lifting of $T_{0,1}X$ gives rise to a natural holomorphic Courant algebroid, as we now show.  
\begin{theorem}\label{courvers}
Let $X$ be a complex manifold. A lifting $D$ of $T_{0,1}X$ to a complex Courant algebroid $E$ gives rise to a natural holomorphic Courant algebroid $E_D$ on $X$. 
\end{theorem}
\begin{proof}
As in the proof of Lemma~\ref{obs}, an isotropic splitting $s:TX\lra E$  gives rise to a 3-form $h_s$ and an isomorphism of Courant algebroids $s_*:(E,[\cdot,\cdot])\lra (\TT X, [\cdot,\cdot]_{h_s})$.  Using this splitting, the lifting $D\subset E$ of $T_{0,1}X$ is given by a 2-form $\theta_s\in\Omega^{1,1}(X)\oplus \Omega^{0,2}(X)$.  By the above remark, we also see that the 3-form $h_s^{3,0}+h_s^{2,1}+\del\theta_s^{1,1}$ is a cocycle and therefore defines a holomorphic Courant algebroid on $\TT_{1,0}X=T_{1,0}X\oplus T^*_{1,0}X$ via the construction in Example~\ref{genex}, taking the cocycle $(T,H)$ in that example to be $T_s=h_s^{2,1}+\del\theta_s^{1,1}$ and $H_s = h_s^{3,0}$. Finally, the equivariance described in the above remark proves that the induced holomorphic Courant algebroid structure on $E_D = \TT_{1,0} X$ varies functorially with the choices.
\end{proof}
To obtain the holomorphic Courant algebroid described above in a more direct way, we use the reduction procedure for Courant algebroids described in~\cite{MR2323543}. 
The lifting $D\subset E$ of $T_{0,1}X$ defines an ``extended action'' of $T_{0,1}X$ on $E$, and we perform a generalization of the symplectic quotient construction for the Courant algebroid $E$, as follows.

The reduction of $E$ by $D$ is given as an orthogonal bundle by $E_D = {D^\bot}/{D}$, where $D^\bot$ is the orthogonal complement of $D$ with respect to the symmetric pairing on $E$. Note that since $D$ is a lifting of $T_{0,1}X$, the kernel of $\pi|_{D^\bot}$ is $D^\bot\cap (T^*\otimes\CC) = T^*_{1,0}X$, and therefore $E_D=D^\bot/D$ is an extension of the form 
\begin{equation*}
\xymatrix{0\ar[r]& T^*_{1,0}X\ar[r]& E_D\ar[r] &T_{1,0}\ar[r] & 0}.
\end{equation*}
The holomorphic structure on $E_D$ is a natural consequence of the general fact that the bundle $D^\bot/D$ inherits a flat connection over the Lie algebroid $D$: given $s\in \Gamma^\infty(X,E_D)$, we define 
\begin{equation}\label{courdelbar}
\delbar_{X} s := [\tilde X, \tilde s]\mod D,
\end{equation}
where $X\in\Gamma^\infty(X,T_{0,1}X)$, $\tilde X$ is the unique lift of $X$ to a section of $D$, and $\tilde s$ is any lift of $s$ to a section of $D^\bot$. The Jacobi identity for the Courant bracket implies that it induces a Courant bracket on the holomorphic sections of $E_D$. In this way, we are able to describe the map 
\begin{equation*}
H^3(M,\CC)\ni[E]\stackrel{D}{\longmapsto} [E_D]\in H^1(\Omega^{2,\cl}(X)),
\end{equation*}
without choosing splittings.  Our final task in this section is to obtain the analogous result for $\CC^*$-gerbes.

\begin{theorem}\label{holconn}
Let $(G,\nabla)$ be a smooth $\CC^*$-gerbe with $0$-connection over a complex manifold $X$, and let $D\subset E_\nabla$ be a lifting of $T_{0,1}X$ to the complex Courant algebroid associated to $\nabla$.  Then $G$ inherits a holomorphic structure.  Furthermore, $G$ inherits a canonical holomorphic $0$-connection $\del$.
\end{theorem}
\begin{proof}
By Theorem~\ref{diracflat}, the presence of $D\subset E_\nabla$ immediately endows $G$ with a flat $D$-connection. Since $D$ is isomorphic to $T_{0,1}X$, the gerbe $G$ is endowed with a holomorphic structure by Theorem~\ref{cangerb}. What remains is to show $G$ inherits a holomorphic $0$-connection.  

Choose a local trivialization in which the gerbe with $0$-connection is given by $\{L_{ij}, \theta_{ijk}, \nabla_{ij}\}$, the Courant algebroid $E_\nabla$ is given as in Theorem~\ref{corresp}, and $D_i=D|_{U_i}$ is given by the graph of $\theta_i\in\Omega^{(1,1)+(0,2)}(U_i)$, so that involutivity is the condition 
\begin{equation*}\label{invthet}
(d\theta_i)^{(1,2)+(0,3)}=0.
\end{equation*}
Since $F_{\nabla_{ij}}$ must glue $D_i$ to $D_j$, we have 
\begin{equation*}
(F_{\nabla_{ij}})^{(1,1)+(0,2)} = \theta_j-\theta_i.
\end{equation*}  
Refining the cover if necessary, choose $\alpha=\{\alpha_i\in\Omega^{(1,0)+(0,1)}(U_i)\}$ such that
\begin{equation*}
(d\alpha_i)^{(1,1)+(0,2)} = \theta_i.
\end{equation*}
Changing the local trivialization by the local line bundles with connection $(U_i\times\CC, d + \alpha_i)$, the $0$-connection has the expression $\nabla_{ij} +\alpha_i - \alpha_j$, which has curvature of type $(2,0)$. This defines a holomorphic gerbe with holomorphic $0$-connection $(G_\alpha,\del_\alpha)$, which a priori depends on $\alpha$. But two choices $\alpha,\alpha'$ of potential for $\{\theta_i\}$, as above, give rise naturally to the local holomorphic line bundles $L_i := (U_i\times\CC, \delbar + {\alpha'}_i^{0,1}-\alpha_i^{0,1})$, with holomorphic connections given by $\del_i:=\del + {\alpha'}_i^{1,0}-\alpha_i^{1,0}$.  The local holomorphic line bundles with holomorphic connections $(L_i,\del_i)$ then define an equivalence 
\begin{equation*}
(L_i,\del_i): (\GG_\alpha,\del_\alpha)\lra (\GG_{\alpha'},\del_{\alpha'}).
\end{equation*}  
The verification that the resulting holomorphic gerbe with $0$-connection is independent of the choices made is similar to that in the proof of Theorem~\ref{cangerb}.
\end{proof}
\begin{remark}
A $1$-connection on a gerbe with $0$-connection $\nabla$ is a maximal isotropic splitting of the Courant algebroid $E_\nabla$; for this reason we may view the lifting $D\subset E_\nabla$ of the Theorem as a partial $1$-connection on the gerbe. 
\end{remark}

\begin{example}
Consider the Hopf surface $X$ from Example~\ref{hopf}, viewed as an elliptic fibration over $\CC P^1$ via the map $(x_1,x_2)\mapsto[x_1:x_2]$. Choose affine charts $(U_0,z_0)$, $(U_1,z_1)$ for the base $\CC P^1$, and write $X$ as the gluing of $(z_0,w_0)\in U_0\times(\CC^*/\ZZ)$ to $(z_1,w_1)\in U_1\times(\CC^*/\ZZ)$  by the map
\begin{equation*}
(z_0,w_0)\longmapsto (1/z_0, z_0w_0).
\end{equation*}
On $U_0\cap U_1$ we have the following real $2$-form
\begin{equation*}
F_{01} = \tfrac{-1}{4\pi}\left(\frac{dz_0\wedge dw_0}{z_0w_0} + \frac{d\bar z_0\wedge d\bar w_0}{\bar z_0\bar w_0}\right).
\end{equation*}
The only nonvanishing period of this 2-form is for the cycle $S^1\times S^1\subset \CC^*\times\CC^*$, which yields $\tfrac{-1}{4\pi}2(2\pi i)(2\pi i)=2\pi$.  Since $F_{01}$ is integral, we may ``prequantize'' it, viewing it as the curvature of a Hermitian line bundle $(L_{01},h_{01},\nabla_{01})$ with unitary connection $\nabla_{01}$. This defines the structure of a Hermitian gerbe with $0$-connection $\nabla$ over $X$, such that the associated Courant algebroid $E_\nabla$ is precisely that from Example~\ref{hopf2}.

To describe a lifting of $T_{0,1}X$ to the Courant algebroid $E_\nabla$, it is convenient to choose a $1$-connection
\begin{equation*}
B_i = \tfrac{1}{4\pi}(\del K_i\wedge\del\log w_i + \delbar K_i\wedge\delbar\log \bar w_i),
\end{equation*}
where $K_i=\log(1+z_i\bar z_i)$ are the usual K\"ahler potentials for the Fubini-Study metric on $\CC P^1$. Computing the global real 3-form $H=dB_i$, we obtain the $(1,2)+(2,1)$-form
\begin{equation*}
H = \tfrac{-1}{8\pi}dd^cK_i\wedge d^c\log (w_i\bar w_i).
\end{equation*}
Observe that $H = d^c\omega$, for the $(1,1)$-form
\begin{equation}\label{omeg}
\omega = \tfrac{-i}{4\pi}(\delbar K_0\wedge \del\log w - \del K_0\wedge\delbar\log \bar w),
\end{equation}
with the significance that $H^{1,2} = \delbar(i\omega)$, which is precisely the condition~\eqref{potent} that $i\omega$ defines a lifting of $T_{0,1}X$.  As a consequence of choosing $\omega$, we obtain a canonical holomorphic structure on the gerbe, as follows. Returning to the \v{C}ech description, the lifting defined by $\omega$ is described by the local forms
\begin{equation*}
\theta_i = B_i^{0,2} - i\omega|_{U_i}.
\end{equation*}
Our open cover is such that $\theta_i^{0,2}$ is $\delbar$-exact, namely
\begin{equation*}
\theta_i^{0,2}=\delbar(\tfrac{1}{4\pi}K_i\wedge\delbar\log\bar w_i).
\end{equation*}
Following Theorem~\ref{holconn}, we perform a gauge transformation by $a=\{a_i = \tfrac{1}{4\pi} K_id\log(w_i\bar w_i)\}$; the new unitary connection $\nabla^a_{01}=\nabla_{01} + a_0-a_1$ has curvature of type $(1,1)$ given by
\begin{equation*}
F^a_{01} = \tfrac{1}{4\pi}\left(\frac{dz_0\wedge d\bar w_0}{z_0\bar w_0} + \frac{d\bar z_0\wedge d w_0}{\bar z_0w_0}\right),
\end{equation*}
so that $\nabla^a_{0,1}$ is indeed a holomorphic structure on the gerbe. After the gauge transformation, the lifting is described by 
\begin{equation*}
\theta^a_i = \theta_i -(da_i)^{(1,1)+(0,2)}= \tfrac{-1}{2\pi}\delbar K_i\wedge\del\log w_i.
\end{equation*}
While $\theta_i^a$ is $\delbar$-closed, it is not exact; therefore, to explicitly describe the holomorphic $0$-connection on the gerbe we would need to refine the cover.  Nevertheless, the associated holomorphic Courant algebroid may be easily constructed; by the prescription in Theorem~\ref{courvers}, it is given by the following holomorphic $(2,0)$-form:
\begin{align*}
\mathcal{B}_{01} &= (B_1^{2,0}+\del a_1^{1,0} - B_0^{2,0}-\del a_0^{1,0})\\
&=\tfrac{-1}{2\pi}(z_0w_0)^{-1}dz_0\wedge dw_0.
\end{align*}
In this way, we recover the holomorphic Courant algebroid studied in Example~\ref{hopf}.
\end{example}

\section{Generalized K\"ahler geometry}\label{Generalizedkahlergeometry}

In K\"ahler geometry, a complex structure $I$ is required to be compatible with a Riemannian metric $g$ in such a way that the 2-form $\omega= gI$ defines a symplectic structure.  The introduction of a Riemannian metric may be thought of as a reduction of structure for $TM$; the complex structure provides a $GL(n,\CC)$ structure which is then reduced by $g$ to the compact Lie group $U(n)$. 

A generalized complex structure on a real Courant algebroid $E$ reduces the usual orthogonal structure $O((n,n),\RR)$ of this bundle to the split unitary group $U(n,n)$.  Generalized K\"ahler geometry may be viewed as an integrable reduction of this structure to its maximal compact subgroup $U(n)\times U(n)$, by the choice of a compatible generalized metric.  

\subsection{Generalized complex and Dirac geometry}\label{gcdg}
Let $(E,\pi,q,[\cdot,\cdot])$ be an exact real Courant algebroid over the smooth manifold $M$. 
\begin{defn}
A generalized complex structure $\JJ$ is an orthogonal bundle endomorphism of $E$, such that $\JJ^2 = -1$, and whose $+i$ eigenbundle $L\subset E\otimes\CC$ is involutive.
\end{defn}
The endomorphism $\JJ$ may preserve the subbundle $T^*M\subset E$, as we saw in Example~\ref{gcbab}, in which case it induces a complex structure on the underlying manifold.  Note, however, that $\JJ$ is not required to preserve the structure of $E$ as an extension; indeed $\JJ(T^*M)$ may be disjoint from $T^*M$, in which case $\JJ(T^*M)$ provides a splitting of $\pi:E\lra TM$ with isotropic and involutive image, and therefore an isomorphism $E\isom\TT M$, with $\JJ$ necessarily of the form 
\begin{equation}\label{jsym}
\JJ_\omega=\begin{pmatrix}0&-\omega^{-1}\\\omega & 0\end{pmatrix},
\end{equation}
for $\omega:TM\lra T^*M$ a symplectic form.  In general, $\JJ(T^*M)$ is a maximal isotropic, involutive subbundle (a Dirac structure) whose intersection with $T^*M$ may vary over the manifold.  Indeed, $Q=\pi\circ\JJ|_{T^*M}:T^*M\lra TM$ is a real Poisson structure controlling the local behaviour of the geometry, in the sense that near a regular point of $Q$, $\JJ$ is isomorphic to the product of a complex and a symplectic structure~\cite{Gualtieri:qr}.  
\begin{example}A particularly illustrative example of a generalized complex structure is furnished by a holomorphic Poisson structure $\sigma$ on a complex manifold $(M,I)$.  This is given by a holomorphic bivector field $\sigma$ with vanishing Schouten bracket $[\sigma,\sigma]$.  Such a Poisson structure determines the following generalized complex structure on the standard Courant algebroid $E=\TT M$:
\begin{equation}\label{holpoi}
\JJ_{\sigma}:=\begin{pmatrix}
I & Q\\
0 & -I^*
\end{pmatrix},
\end{equation}
where $Q$ is the imaginary part of $\sigma = P  + iQ$.  The peculiar aspect of the generalized complex structure $\JJ_\sigma$ is that the complex structure obtained from its action on $TM\subset\TT M$ is not intrinsic. Indeed, a symmetry of the Courant algebroid, such as is given by a closed 2-form $B\in\Omega^{2,\cl}(M)$, conjugates~\eqref{holpoi} into 
\begin{equation*}
e^B\JJ_\sigma e^{-B}
=
\begin{pmatrix}
I-QB & Q\\
I^*B+BI-BQB & BQ-I^*
\end{pmatrix}.
\end{equation*}
In fact, in~\cite{Gualtieri:2007fe} it is shown that in some cases, $B$ may be chosen so that $I^*B+BI-BQB$ vanishes, rendering $e^B\JJ_\sigma e^{-B}$ again into the form~\eqref{holpoi}, but for a different complex structure $J=I-QB$. For example, this relation exists between the second Hirzebruch surface $\mathbb{F}_2$ and $\CC P^1\times\CC P^1$, which support holomorphic Poisson structures that are isomorphic as generalized complex structures. 
\end{example}
The $\pm i$ eigenbundles $L,\bar L\subset E\otimes\CC$ of a generalized complex structure define two Dirac structures which are transverse, in the sense $L\cap \bar L = \{0\}$.  In such a situation, as shown in~\cite{MR1472888}, the two Lie algebroids defined by $L,\bar L$ enjoy a compatibility condition, making them into a Lie bialgebroid.  Identifying $\bar L$  with $L^*$ using the symmetric pairing, this means that the Lie bracket on $\bar L$ may be extended to a Schouten bracket on the sheaf of graded algebras $\Omega^\bullet_L= \Gamma^\infty(\wedge^k L^*)$, and that the Lie algebroid differential $d_L$ on this algebra is a graded derivation of the bracket. In summary, we obtain a sheaf of differential graded Lie algebras from the transverse Dirac structures $(L,\bar L)$.
\begin{equation*}
(L,\bar L)\leadsto(\Omega^\bullet_L, d_L, [\cdot,\cdot])
\end{equation*} 
In~\cite{Gualtieri:qr}, it is shown that the above differential graded Lie algebra is elliptic, so that on a compact manifold, it has finite-dimensional cohomology groups $H^\bullet_{L}$. Furthermore, it controls the deformation problem for generalized complex structures, so that the vanishing of an obstruction map
\begin{equation*}
\Phi:H^2_L\lra H^3_L
\end{equation*}
ensures a smooth local moduli space modeled on $H^2_L$.   

A pair of transverse Dirac structures, as $(L,\bar L)$ above, enjoy a further transversality property, which we now describe.  We first recall the Baer sum operation on Courant algebroids~\cite{sevwein,MR2171584,Gualtieri:qr}, which is a realization of the additive structure on $H^1(\Omega^{2,\cl})$.
\begin{defn}\label{baercour}
Given two Courant algebroids $E_1,E_2$ on $M$ with projections $\pi_i:E_i\mapsto TM$, their Baer sum as extensions of $TM$ by $T^*M$, namely the bundle 
\begin{equation*}
E_1\boxtimes E_2 := (E_1\oplus_{TM} E_2)/K
\end{equation*}
for $K=\{(-\pi_1^*\xi,\pi_2^*\xi): \xi\in T^*\}$, carries a natural Courant algebroid structure, defined componentwise.
\end{defn}  
The standard Courant algebroid on $\TT M$ is the identity element for the Baer sum, and the inverse of a Courant algebroid $(E,\pi,q,[\cdot,\cdot])$, called the transpose $E^\top$, is given simply by reversing the symmetric inner product: $E^T = (E,\pi,-q,[\cdot,\cdot])$. The Baer sum operation may also be applied to Dirac structures, as explained in~\cite{ABM,Gualtieri:qr}: 
\begin{prop}
If $D_1\subset E_1$ and $D_2\subset E_2$ are Dirac structures which are transverse over $TM$, in the sense that $\pi_1(D_1) + \pi_2(D_2) = TM$, then their Baer sum 
\begin{equation*}
D_1\boxtimes D_2 := \frac{(D_1\oplus_{TM} D_2) + K}{K},
\end{equation*}
for $K$ as in Definition~\ref{baercour}, is a Dirac structure in $E_1\boxtimes E_2$.
\end{prop}
A pair of Dirac structures $D_1,D_2$ such that $D_1\cap D_2=\{0\}$ are transverse in the above sense, and we may form their Baer sum $D_1^\top\boxtimes D_2\subset \TT M$, where $D_1^\top$ is simply $D_1$ viewed as a Dirac structure in $E^\top$.  From observations made in~\cite{MR1262213}, it follows that this Baer sum is given by 
\begin{equation}\label{transdirac}
D_1^\top\boxtimes D_2 = \Gamma_\beta\subset \TT M,
\end{equation}
where $\Gamma_\beta$ is the graph of a Poisson structure $\beta$ (see~\cite{ABM} for a proof).  Recall that, as a Lie algebroid, $\Gamma_\beta\isom T^*M$ has bracket given by $[df,dg] = d(\beta(df,dg))$, and anchor map $\beta:T^*M\lra TM$.  The Poisson structure $\beta$ may also be described in the following way: let $P_{D_i}: E\lra D_i$ be the projection operators for the direct sum $E = D_1\oplus D_2$. Then $\beta$ is given by
\begin{equation}\label{poistr}
\beta = \pi\circ P_{D_1}|_{T^*M}:T^*M\lra TM.
\end{equation}
The geometry induced by such a pair of transverse Dirac structures $(D_1,D_2)$ in $E$ may be understood in the following way: by projection to $TM$, $D_1$ and $D_2$ each induce singular foliations $\FF_1,\FF_2$ on the manifold $M$.  The transversality condition on the $D_i$ implies that the induced foliations are transverse, in the sense $T\FF_1+T\FF_2 = TM$.   
As shown in~\cite{Gualtieri:qr}, the exact Courant algebroid $E$ may be pulled back to any submanifold $\iota:S\hookrightarrow M$, yielding an exact Courant algebroid over $S$, defined by
\begin{equation}\label{redsub}
E_S:= K^\perp/K,
\end{equation}
where $K= N^*S\subset E|_S$ is the conormal bundle of $S$. If $S$ happens to be a leaf of the singular foliation $\FF$ induced by a Dirac structure $D\subset E$, then the Dirac structure also pulls back, yielding an isotropic, involutive splitting $s_D$ of $\pi:E_S\lra TS$.  Therefore, in the presence of two transverse Dirac structures $(D_1,D_2)$, if we pull back $E$ to a leaf $S$ of the singular foliation induced by $D_1^\top\boxtimes D_2$, it will have two splittings $s_{D_1}, s_{D_2}$, each obtained from one of the Dirac structures.  The resulting splittings are themselves transverse in $E_S$, and therefore they differ by a section $\omega_S\in\Omega^{2,\cl}(S)$ which must be nondegenerate.  This is precisely the symplectic form on the leaf of the Poisson structure $\beta$. 

Finally, we wish to emphasize an algebraic implication of the Baer sum identity described above.  
\begin{prop}\label{baerfiber}
Let $D_1,D_2\subset E$ be Dirac structures such that $D_1\cap D_2 = \{0\}$.  Then the Baer sum $D_1^\top\boxtimes D_2 = \Gamma_\beta$, coincides with the fiber product of the Lie algebroids $D_1,D_2$ over $TM$.  As a result, we have the isomorphism of sheaves of differential graded algebras
\begin{equation}\label{tensdir}
(\wedge^\bullet \T_M, d_\beta) = (\Omega^\bullet_{D_1}, d_{D_1})\otimes_{\Omega^\bullet_T} (\Omega^\bullet_{D_2}, d_{D_2}), 
\end{equation}
where $(\wedge^\bullet \T_M, d_\beta = [\beta,\cdot])$ is the Lichnerowicz complex\footnote{The hypercohomology of this complex is the well-known Poisson cohomology of $\beta$.} of sheaves of multivector fields induced by the Schouten bracket with the Poisson structure $\beta$, and the anchor maps $\pi_i:D_i\lra TM$ induce the morphisms $\pi_i^*:\Omega^\bullet_T\lra \Omega^\bullet_{D_i}$ from the usual de Rham complex of $M$, which are used in the tensor product.
\end{prop}
\begin{proof}
This follows from the simple observation that $K=\{(-\pi_1^*\xi,\pi_2^*\xi)\ :\ \xi\in T^*M\}$ intersects $D_1\oplus_{TM} D_2= \{(d_1,d_2)\in D_1\oplus D_2\ :\ \pi_1(d_1)=\pi_2(d_2)\}$ precisely in $D_1\cap D_2\cap T^*M$, which vanishes since $D_1\cap D_2 = \{0\}$.  In this way, $D_1^{\top}\boxtimes D_2$ coincides, as a Lie algebroid, with the fiber product of $D_1$,$D_2$ as Lie algebroids, yielding the diagram of Lie algebroids
\begin{equation*}
\xymatrix{\Gamma_\beta\ar[r]\ar[d]& D_2\ar[d]\\D_1\ar[r] & TM}
\end{equation*}
which dualizes to the fact that the Lichnerowicz complex is given by the (graded) tensor product~\eqref{tensdir}. 
\end{proof}
This is of particular importance when studying modules over the Lie algebroids $D_1$ or $D_2$, i.e. vector bundles (or sheaves of $\OO_M$-modules) with flat $D_i$-connections as in Definition~\ref{Acon}.
\begin{corollary}\label{transverpoi}
For a pair of transverse Dirac structures $(D_1,D_2)$, the tensor product of a $D_1$-module and a $D_2$-module is a $\Gamma_\beta$-module, i.e. a Poisson module~\cite{MR1423640, MR1465521}. In particular, any $D_i$-module is also a Poisson module.
\end{corollary}
In the case of a generalized complex structure, we have the transverse Dirac structures $(L,\overline{L})$, and it was shown in~\cite{Gualtieri:qr} that their Baer sum is
\begin{equation}\label{relpoi}
L^\top\boxtimes \overline{L} = \Gamma_{iQ/2}\subset \TT M,
\end{equation}
for the Poisson structure $Q=\pi\circ\JJ|_{T^*M}$ described earlier.  A vector bundle with a flat $\overline{L}$-connection on a generalized complex manifold is called a generalized holomorphic bundle~\cite{Gualtieri:qr}, hence we have the following consequence of the above Baer sum:
\begin{corollary}
Any generalized holomorphic bundle inherits a Poisson module structure, for the underlying real Poisson structure. 
\end{corollary} 

\subsection{Generalized K\"ahler structures}

Let $(E,\pi,q,[\cdot,\cdot])$ be an exact real Courant algebroid over the smooth manifold $M$. 
\begin{defn}
A generalized K\"ahler structure on $E$ is a pair $(\JJ_+,\JJ_-)$ of generalized complex structures on $E$ which commute, i.e. $\JJ_+\JJ_- = \JJ_-\JJ_+$, and such that the symmetric pairing 
\begin{equation}\label{gkah}
G(x,y) := \IP{\JJ_+ x, \JJ_- y}
\end{equation}
is positive-definite, defining a metric on $E$ called the generalized K\"ahler metric.
\end{defn}
A usual K\"ahler structure on a manifold is given by a complex structure $I$ compatible with a Riemannian metric $g$, in the sense that $\omega:=gI$ is a symplectic form.  This defines the following generalized K\"ahler structure on $\TT M = TM\oplus T^*M$:
\begin{equation}\label{usukah}
\JJ_+ = \begin{pmatrix}0 & \omega^{-1}\\\omega& 0\end{pmatrix},\ \ \ \ 
\JJ_- = \begin{pmatrix}I & 0\\0&-I^*\end{pmatrix},
\end{equation}
so that $G(X+\xi,Y+\eta) = \tfrac{1}{2}(g(X,Y) + g^{-1}(\xi,\eta))$ is the usual K\"ahler metric.
The generalized K\"ahler metric~\eqref{gkah} is an example of a generalized metric, which may be viewed as a reduction of structure for the Courant algebroid $E$, from its usual $O(n,n)$ structure to the compact form $O(n)\times O(n)$. 
\begin{defn}
A generalized metric $G$ on $E$ is a positive-definite metric on $E$ which is compatible with the pre-existing symmetric pairing $\IP{\cdot,\cdot}$, in the sense that it is obtained by choosing a maximal positive-definite subbundle $C_+\subset E$ (with orthogonal complement $C_-:=C_+^\bot$), and defining 
\begin{equation*}
G(x,y):=\IP{x_+,y_+} - \IP{x_-,y_-},
\end{equation*} 
where $x_\pm$ denotes the orthogonal projection to $C_\pm$.
\end{defn}
Identifying $E$ with $E^*$ using $\IP{\cdot,\cdot}$, we may view $G$ as a self-adjoint endomorphism $G:E\lra E$, with $\pm 1 $ eigenbundle given by $C_\pm$.  For a generalized K\"ahler structure, $G= -\JJ_+\JJ_-$, so that for~\eqref{usukah}, $G$ is given by 
\begin{equation*}
G =\begin{pmatrix}0 & g^{-1}\\g & 0\end{pmatrix}.
\end{equation*}
\begin{example}
Any positive-definite subbundle $C_+\subset\TT M$ is the graph of a bundle map $\theta:TM\lra T^*M$ with positive-definite symmetric part. That is, $\theta = b + g$, with $b\in\Omega^2(M)$ and $g$ a Riemannian metric.  Then $C_\pm$ is the graph of $b\pm g$.   The corresponding endomorphism $G:E\lra E$ is then described by 
\begin{equation*}
G_{g,b} =e^b\begin{pmatrix}0 & g^{-1}\\g & 0\end{pmatrix}e^{-b}=\begin{pmatrix}-g^{-1} b&g^{-1}\\g - bg^{-1} b
&bg^{-1}\end{pmatrix}.
\end{equation*}
Note that the induced Riemannian metric on $TM\subset \TT M$ is $g - bg^{-1}b$, while the metric on $T^*M$ is the usual inverse metric of $g$.
\end{example}
An immediate consequence of the choice of generalized metric on $E$ is that the projection $\pi:E\lra TM$ obtains two splittings $s_\pm$ corresponding to the two subbundles $C_\pm$, simply because the definite bundles $C_\pm$ intersect the isotropic subbundle $T^*M\subset E$ trivially.  The average of these splittings, $s = \tfrac{1}{2}(s_+ + s_-)$, is then a splitting of $\pi$ with isotropic image $G(T^*M)$.  The splitting induces an isomorphism $s_*: E\lra \TT M$ which sends the definite subbundles $C_\pm$ to the graphs $\Gamma_{\pm g}$, for $g$ a Riemannian metric on $M$.  In summary, we have the following:
\begin{prop}
The choice of a generalized metric $G$ on $E$ is equivalent to a choice of isotropic splitting $s:TM\lra E$, together with a Riemannian metric $g$ on $M$, such that 
\begin{equation}\label{splits}
C_\pm = \{s(X) \pm g(X)\ :\ X\in TM\}.
\end{equation} 
\end{prop}
As a result of the splitting $s$ determined by the generalized metric, we immediately obtain a closed 3-form $H\in\Omega^{3,\cl}(M,\RR)$, as in~\eqref{hache}, which twists the Courant bracket on $\TT M$ as in~\eqref{twcour}, so that $s$ induces an isomorphism of Courant algebroids $s_*:E\lra (\TT M, [\cdot,\cdot]_H)$.
\begin{defn}
The \emph{torsion} of a generalized metric $G$ on the Courant algebroid $E$ is the 3-form $H$ corresponding to the splitting of $E$ defined by $G(T^*M)$.
\end{defn}
\begin{corollary}
If $E$ is the Courant algebroid obtained from a Hermitian gerbe with unitary connection, a choice of generalized metric induces a 1-connection with curvature given by the torsion $H$ of the generalized metric.
\end{corollary}
We wish to describe the geometric structures induced on $M$ by the generalized K\"ahler pair $(\JJ_+,\JJ_-)$.  First, we leave aside questions of integrability and describe \emph{almost} generalized K\"ahler structures, which are generalized K\"ahler structures without the Courant involutivity conditions on $\JJ_+$ and $\JJ_-$. 

An almost generalized complex structure $\JJ_+$ is compatible with the generalized metric $G$ when it preserves $C_+$ (and hence, necessarily, $C_-$), or equivalently, when it commutes with $G$.  This compatibility is also equivalent to the fact that $\JJ_- := G\JJ_+$ is an almost generalized complex structure.  Since $C_\pm$ are the $\pm 1$ eigenbundles of $G$, we have  
\begin{equation}\label{agre}
\JJ_+|_{C_\pm} = \pm \JJ_-|_{C_\pm},
\end{equation}
and so the complex structures on the bundles $C_\pm$ induced by $\JJ_+,\JJ_-$ coincide up to sign.  Using the identifications of metric bundles
\begin{equation*}
s_\pm:(TM,g)\lra (C_\pm,\pm \IP{\cdot,\cdot}),
\end{equation*}
we obtain two almost complex structures $I_+, I_-$ on the manifold $M$, each of which is compatible with the Riemannian metric $g$, hence forming an almost bi-Hermitian structure.  We now show that the correspondence $(\JJ_+,\JJ_-)\mapsto (g,I_+, I_-)$ is an equivalence. 
\begin{theorem}\label{linalg}
An almost generalized K\"ahler structure $(\JJ_+,\JJ_-)$ on $E$ is equivalent to the data $(s,g,I_+,I_-)$, where $s$ is an isotropic splitting of $E$, $g$ is a Riemannian metric, and $I_\pm$ are almost complex structures compatible with $g$.  
\end{theorem} 
\begin{proof}
We have already explained how to extract $(s,g,I_\pm)$ from $(\JJ_+,\JJ_-)$.   To reconstruct $(\JJ_+,\JJ_-)$ from the bi-Hermitian data, we construct definite splittings $s_\pm$ via 
\begin{equation*}
s_\pm := s + g : TM\lra E,
\end{equation*}  
following Equation~\eqref{splits}, and use the fact that $\JJ_+,\JJ_-$ are built from the complex structures $I_\pm$ by transporting them to $C_\pm$ and using Equation~\eqref{agre}:
\begin{equation}\label{prerec}
\JJ_\pm := s_+ I_+ s_+|_{C_+}^{-1} \pm s_- I_- s_-|_{C_-}^{-1},
\end{equation}
which expands to the expression:
\begin{equation}\label{reconstruct}
\JJ_\pm=s_*^{-1}\frac{1}{2}
\begin{pmatrix}I_+\pm I_- & -(\omega_+^{-1}\mp\omega_-^{-1}) \\
\omega_+\mp\omega_-&-(I^*_+\pm I^*_-)\end{pmatrix}s_*,
\end{equation}
where $s_*:E\lra \TT M$ is the isomorphism induced by $s$, and $\omega_\pm:=gI_\pm$ are the nondegenerate 2-forms determined by the almost Hermitian structures $(g,I_\pm)$. The two constructions are easily seen to be mutually inverse.
\end{proof}
Before proceeding to translate the integrability condition from the generalized complex structures to the bi-Hermitian data, we make some comments concerning orientation. 
\begin{remark}The \emph{type}~\cite{Gualtieri:qr} of a generalized complex structure $\JJ$ at a point is defined to be 
\begin{equation*}
\mathrm{type}(\JJ):=\tfrac{1}{2}\mathrm{corank}_\RR(Q),
\end{equation*}
where $Q=\pi\circ\JJ|_{T^*}$ is the real Poisson structure associated to $\JJ$.  On a real $2n$-manifold, the type may vary between $0$, where $\JJ$ defines a symplectic structure, and $n$, where it defines a complex structure.  The parity of the type, however, is locally constant, as it is determined by the orientation induced by $\JJ$ on $E$ ($\det E$ is canonically trivial and $\tfrac{1}{2n!}\det\JJ=+1$ or $-1$ as $\type(\JJ)$ is even or odd, respectively).  If we have an almost generalized K\"ahler structure on a real $2n$-manifold, the equation $G = -\JJ_+\JJ_-$ yields 
\begin{equation*}
\det\JJ_+\det\JJ_- = \det G = (-1)^n,
\end{equation*}
Implying that $\JJ_+$ and $\JJ_-$ must have equal parity in real dimension $4k$ and unequal parity in real dimension $4k+2$.  Furthermore, by Equation~\eqref{prerec}, the parity of $\JJ_+$ is even or odd as the orientations induced by $I_\pm$ agree or disagree, respectively.  This leads immediately to the fact that in real dimension $4k$, both $\JJ_\pm$ may either have even parity, in which case $I_\pm$ induce the same orientation, or odd parity, in which case $I_\pm$ induce opposite orientations on $M$.  In dimension $4k+2$, however, there is no constraint placed on the orientations of $I_\pm$, since $I_+$ may be replaced with $-I_+$ without altering the parity of $\JJ_\pm$. 
\end{remark}
\begin{example}\label{typeorient}
If $\dim_\RR M = 4$, an almost generalized K\"ahler structure may either have $\type(\JJ_+)=\type(\JJ_-)=1$, in which case $I_\pm$ induce opposite orientations, or $\JJ_\pm$ both have even type, in which case $I_\pm$ must induce the same orienation on $M$.
\end{example}

\subsection{Integrability and bi-Hermitian geometry}
Let $(s,g,I_\pm)$ be the almost bi-Hermitian data corresponding to an almost generalized K\"ahler structure $\JJ_\pm$, as in Theorem~\ref{linalg}.  In this section, we describe the integrability conditions on $(s,g,I_\pm)$ corresponding to the integrability of $\JJ_+$ and $\JJ_-$.  

Recall that the integrability condition for $\JJ_\pm$ is that the $+i$-eigenbundles $L_\pm = \ker(\JJ_\pm - i\bf{1})$ are involutive for the Courant bracket on $E\otimes\CC$.  Since $\JJ_\pm$ commute, the eigenbundle of $\JJ_+$ decomposes into eigenbundles of $\JJ_-$, so that 
\begin{equation}\label{defell}
L_+=\ell_+\oplus\ell_-,
\end{equation}
where $\ell_+ = L_+\cap L_-$ and $\ell_- = L_+\cap \overline{L_-}$. Since $G=-\JJ_+\JJ_-$ has eigenvalue $+1$ on $\ell_+\oplus \overline\ell_+$, we also have
\begin{equation}\label{compg}
C_\pm\otimes\CC = \ell_\pm \oplus \overline\ell_\pm.
\end{equation}
As a result, we obtain a decomposition of the Courant algebroid into four isotropic subbundles, each of complex dimension $n$ on a real $2n$-manifold:
\begin{equation}\label{kahldeco}
E\otimes\CC = \ell_+\oplus\ell_-\oplus\overline\ell_+\oplus\overline \ell_-.
\end{equation}
Since $\ell_\pm$ are intersections of involutive subbundles, they are individually involutive.  This is actually an equivalent characterization of the integrability condition on $\JJ_\pm$.
\begin{prop}
The almost generalized K\"ahler structures $\JJ_\pm$ are integrable if and only if both the subbundles $\ell_\pm$ described above are involutive. 
\end{prop}
\begin{proof}
That the integrability of $\JJ_\pm$ implies the involutivity of $\ell_\pm$ is explained above.  Now let $\ell_\pm$ be involutive. We must show that $L_+ = \ell_+\oplus\ell_-$ and $L_- = \ell_+\oplus \overline\ell_-$ are involutive.  To prove that $L_+$ is involutive, we need only show that if $x_\pm$ is a section of $\ell_\pm$, then $[x_+,x_-]$ is a section of $L_+$.  We do this by showing $[x_+,x_-]$ is orthogonal to both $\ell_+$, $\ell_-$, and hence must lie in $\ell_+^\bot\cap\ell_-^\bot = (\ell_+\oplus\ell_-)^\bot = L_+^\bot = L_+$, where we have used the fact that $L_+$ is maximal isotropic.  For $y_\pm$ any section of $\ell_\pm$ we have:  
\begin{align*}
\IP{[x_+,x_-], y_+} &= \pi(x_+)\IP{x_-,y_+} - \IP{x_-,[x_+,y_+]} = 0,\\
\IP{[x_+,x_-], y_-} &= -\IP{[x_-,x_+], y_-} = \pi(x_-)\IP{x_+,y_-} - \IP{x_+,[x_-,y_-]} = 0,
\end{align*}
hence $[x_+,x_-]$ is in $L_+$, as required.  $L_-$ is shown to be involutive in the same way.  
\end{proof}
To understand what this integrability condition imposes on the bi-Hermitian data, we use Theorem~\ref{linalg} to express the bundles $\ell_\pm$ purely in terms of the data $(s,g, I_\pm)$.  By~\eqref{defell} and~\eqref{compg}, we see that $\ell_\pm$ is the $+i$ eigenbundle of $\JJ_+$ acting on $C_\pm\otimes\CC$.  Since the almost complex structures $I_\pm$ are defined via the restriction of $\JJ_+$ to $C_\pm$, we obtain:
\begin{equation}\label{ellexp}
\ell_\pm = \{(s \pm g)X\ :\ X\in T^{1,0}_\pm M\},
\end{equation}
where $T^{1,0}_\pm M$ is the $+i$ eigenbundle of the almost complex structure $I_\pm$. Using the 2-forms $\omega_\pm = g I_\pm$, we obtain a more useful form of Equation~\eqref{ellexp}:
\begin{align}\label{ellexp2}
\ell_\pm &= \{sX \mp i\omega_\pm X\ :\ X\in T^{1,0}_\pm M\}\notag\\
&=e^{\mp i\omega_\pm}s(T^{1,0}_\pm M),
\end{align}
where $e^{\mp i\omega_\pm}$ acts on $x\in E$ via $x\mapsto x + i_{\pi(x)}(\mp i\omega_\pm)$. 
\begin{theorem}
Let $(\JJ_+,\JJ_-)$ be an almost generalized K\"ahler structure, described equivalently by the almost bi-Hermitian data $(s,g,I_+, I_-)$ as above. $(\JJ_+,\JJ_-)$ is integrable if and only if $I_\pm$ are integrable complex structures on $M$, and the following constraint holds: 
\begin{equation}\label{constraint}
\pm d^c_\pm \omega_\pm = H, 
\end{equation}
where $H\in\Omega^{3,\cl}(M,\RR)$ is the closed 3-form corresponding to the section $s$ via Equation~\eqref{hache}, and $d^c_\pm = i(\delbar_\pm - \del_\pm)$ are the real Dolbeault operators corresponding to the complex structures $I_\pm$.
\end{theorem}
\begin{proof}
Using expression~\eqref{ellexp2} for $\ell_\pm$, let $e^{\mp i\omega_\pm}s(X)$, $e^{\mp i\omega_\pm}s(Y)$  be two sections of $\ell_\pm$, where $X,Y$ are vector fields in $T^{1,0}_\pm M$.  Then the properties of the Courant bracket and the definition of $H$ from Equation~\eqref{hache} yield
\begin{align*}
[e^{\mp i\omega_\pm}s(X),e^{\mp i\omega_\pm}s(Y)] &= e^{\mp i\omega_\pm}[s(X), s(Y)] + i_Yi_Xd(\mp i\omega_\pm)\\
&=e^{\mp i\omega_\pm}(s([X,Y]) + i_Xi_Y H) + i_Yi_Xd(\mp i\omega_\pm)\\
&=e^{\mp i\omega_\pm}s([X,Y]) + i_Xi_Y (H \pm id\omega_\pm).
\end{align*}
This is again a section of $\ell_\pm$ if and only if $[X,Y]$ is in $T^{1,0}_\pm M$ and $(H\pm i d\omega_\pm)^{(3,0)+(2,1)}$ vanishes.  The first condition is precisely the integrability of the complex structures $I_\pm$, and in this case since $\omega_\pm$ is of type $(1,1)$ with respect to $I_\pm$, $d\omega$ has no $(3,0)$ component. The second condition is then the statement that 
\begin{equation*}
H^{2,1}  = \mp i \del\omega_\pm,
\end{equation*}
which together with its complex conjugate yields $H = \pm d^c_\pm \omega_\pm$, as required.
\end{proof}
The above theorem demonstrates that generalized K\"ahler geometry, involving a pair of commuting generalized complex structures, may be viewed classically as a bi-Hermitian geometry, in which the pair of usual complex structures need not commute, and with an additional constraint involving the torsion 3-form $H$.  This bi-Hermitian geometry is known in the physics literature: Gates, Hull, and Ro\v{c}ek showed in~\cite{MR776369} that upon imposing $N=(2,2)$ supersymmetry, the geometry induced on the target of a 2-dimensional sigma model is precisely this one. 

\begin{corollary}\label{nonalg}
If the torsion $H$ of a compact generalized K\"ahler manifold has nonvanishing cohomology class in $H^3(M,\RR)$, then the complex structures $I_\pm$ must both fail to satisfy the $dd^c_\pm$-lemma; in particular, they do not admit K\"ahler metrics and are not algebraic varieties. 
\end{corollary}
\begin{proof}
Suppose $I_+$ satisfies the $dd^c_+$-lemma. Then since $H = d^c\omega_+$ and $dH=0$, we conclude there exists $a_+\in\Omega^1(M,\RR)$ with $H = dd^c_+ a_+$, implying $H$ is exact in either case, as required. The same argument holds for $I_-$.
\end{proof}

\subsection{Examples of generalized K\"ahler manifolds}

The main examples of generalized K\"ahler manifolds in the literature were constructed in several different ways: by imposing symmetry~\cite{MR2217300}, by a generalized K\"ahler reduction procedure analogous to symplectic reduction~\cite{MR2323543, MR2397619, MR2249799}, by recourse to twistor-theoretic results on surfaces~\cite{MR1702248, MR2287917}, by a flow construction using the underlying real Poisson geometry~\cite{MR2371181, Gualtieri:2007fe}, and by developing a deformation theory for generalized K\"ahler structures~\cite{Goto:2007if, MR2272873, Goto:2009ix, Goto:2009zm} whereby one may deform usual K\"ahler structures into generalized ones.  We elaborate on some illustrative examples from~\cite{Gualtieri:rp}.

\begin{example}\label{hkex}
Let $(M,g,I,J,K)$ be a hyper-K\"ahler structure. Then clearly
$(g,I,J)$ is a bi-Hermitian structure, and since
$d\omega_I=d\omega_J=0$, we see that $(g,I,J)$ defines a
generalized K\"ahler structure for the standard Courant structure on $\TT M$.  From formula~(\ref{reconstruct}), we reconstruct the generalized
complex structures:
\begin{equation}\label{HKbih}
\JJ_{\pm}=\frac{1}{2}
\left(\begin{matrix}I\pm J & -(\omega_I^{-1}\mp\omega_J^{-1}) \\
\omega_I\mp\omega_J&-(I^*\pm J^*)\end{matrix}\right).
\end{equation}
Note that that~(\ref{HKbih}) describes two generalized complex
structures of symplectic type, a fact made manifest via the following expression:
\begin{align*}
\JJ_{\pm}&=e^{\pm\omega_K}
\left(\begin{matrix} 0& -\tfrac{1}{2}(\omega_I^{-1}\mp\omega_J^{-1}) \\
\omega_I\mp\omega_J&0\end{matrix}\right)
e^{\mp\omega_K}.
\end{align*}
The same observation holds for any two non-opposite complex structures $I_1,I_2$ in the 2-sphere of hyper-K\"ahler complex structures, namely that the bi-Hermitian structure given by $(g,I_1,I_2)$ defines a generalized K\"ahler structure where both generalized complex structures are of symplectic type. 
\end{example}

The bi-Hermitian structure obtained from a hyperk\"ahler
structure is an example of a \emph{strongly} bi-Hermitian
structure in the sense of~\cite{MR1702248}, i.e. a bi-Hermitian structure such that $I_+$ is nowhere equal to $\pm I_-$.  From expression~\eqref{reconstruct}, it is clear that in 4 dimensions, strongly
bi-Hermitian structures with equal orientation correspond exactly
to generalized K\"ahler structures where both generalized complex
structures are of symplectic type.

\begin{example}
The generalized K\"ahler structure described in Example~\ref{hkex} can be deformed using a method similar to that described in~\cite{Gualtieri:2007fe}.  The complex 2-form $\sigma_I=\omega_J+i\omega_K$ on a hyper-K\"ahler structure is a holomorphic symplectic form with respect to $I$, and similarly $\sigma_J=-\omega_I + i\omega_K$ is holomorphic symplectic with respect to $J$.  As in Equation~\eqref{holpoi}, these define generalized complex structures on $\TT M$ given by:
\begin{equation}
\JJ_{\sigma_I}:=\begin{pmatrix}
I & \omega_K^{-1}\\
0 & -I^*
\end{pmatrix},\ \ \ \ \ 
\JJ_{\sigma_J}:=\begin{pmatrix}
J & \omega_K^{-1}\\
0 & -J^*
\end{pmatrix}.
\end{equation}
Interestingly, the symmetry $e^F$, for the closed 2-form $F=\omega_I+\omega_J$, takes $\JJ_{\sigma_I}$ to $\JJ_{\sigma_J}$, so that 
\begin{equation*}
e^F\JJ_{\sigma_I}e^{-F}=\JJ_{\sigma_J}.
\end{equation*}
Now choose $f\in C^\infty(M,\RR)$ and let $X_f$ be its Hamiltonian vector field for the Poisson structure $\omega_K^{-1}$.  Let $\varphi_t$ be the flow generated by this vector field, and define  
\begin{equation*}
F_t(f):= \int_0^t\varphi^*_s(dd^c_J f)ds.
\end{equation*}
In~\cite{Gualtieri:2007fe}, it is shown that the symmetry $e^{F_t}$ takes $\JJ_{\sigma_J}$ to the deformed generalized complex structure
\begin{equation*}
\JJ_{\sigma_{J_t}}=\begin{pmatrix}
J_t & \omega_K^{-1}\\
0 & -J^*_t
\end{pmatrix},
\end{equation*}
where $J_t= \varphi^*_t(J)$, and that as a result, $I$ and $J_t$ give a family of generalized K\"ahler structures with respect to the deformed metric
\begin{equation*}
g_t = -\tfrac{1}{2}(F+F_t(f))(I+J_t),
\end{equation*}
where $g_0=-\tfrac{1}{2}F(I+J)$ is the original hyper-K\"ahler metric. 

The idea of deforming a Hyperk\"ahler structure to obtain a bi-Hermitian structure appeared, with a different formulation, in~\cite{MR1702248} (see also~\cite{MR2217300}), where it is shown for surfaces
that the Hamiltonian vector field can be chosen so that the
resulting deformed metric is not anti-self-dual, and hence by a result in~\cite{Pontecorvo}, cannot admit more than two distinct orthogonal complex structures.
\end{example}

\begin{example}[The Hopf surface: odd generalized K\"ahler]\label{hopfodd}
Let $X$ be the standard Hopf surface from Example~\ref{hopf}, and denote its complex structure by $I_-$.  The product metric on $S^3\times S^1$ can be written as follows:
\begin{equation}\label{Hopfmetric}
g=\tfrac{1}{4\pi R^{2}}(dx_1 d\bar x_1 + dx_2d\bar x_2),
\end{equation}
for $R^2 = x_1\bar x_1 + x_2\bar x_2$.  The complex structure $I_-$ is manifestly Hermitian for this metric, with associated 2-form $\omega_-=gI_-$ given by:
\begin{equation*}
\omega_-=\tfrac{i}{4\pi R^2}(dx_1\wedge d\bar x_1 + dx_2\wedge d\bar x_2),
\end{equation*}
and its complex derivative $H=-d^c\omega_-$ is a real closed 3-form on $X$ generating $H^3(X,\ZZ)$.

Now let $I_+$ be the
complex structure on the Hopf surface obtained by modifying the complex
structure on $\CC^2$ such that $(x_1,\overline{x_2})$ are holomorphic coordinates; note that $I_\pm$ have opposite
orientations, and are both Hermitian with respect to $g$. Also, it
is clear that
$d^c_{+}\omega_{+}=-d^c_-\omega_-=H$. Therefore, the bi-Hermitian data $(g, I_\pm)$ defines a generalized K\"ahler structure on $(\TT X, H)$, the standard Courant algebroid twisted by $H$.  Since $I_\pm$ induce opposite orientations, the corresponding generalized complex structures $\JJ_\pm$ are both of odd type, by Example~\ref{typeorient}.  Note that the complex structures
$I_\pm$ happen to commute, a special case studied in~\cite{MR2287917}.    This particular generalized K\"ahler geometry first appeared in the context of a supersymmetric $SU(2)\times U(1)$ Wess-Zumino-Witten model~\cite{RocekSchoutensSevrin}.
\end{example}

\begin{example}[The Hopf surface: even generalized K\"ahler]\label{evenhopf}
Let $(g,I_-)$ be the standard Hermitian structure on the Hopf surface, as in Example~\ref{hopfodd}.  We specify a different complex structure $I_+$ by providing a generator $\Omega_+\in \Omega^{2,0}_+(X)$, namely:
\begin{equation}\label{iplushopf}
\Omega_+:=\tfrac{1}{R^4}(\bar x_1 dx_1+x_2d\bar x_2)\wedge(\bar x_1 dx_2-x_2d\bar x_1).
\end{equation}
Comparing this with the usual complex structure, where the generator is given by $\Omega_- = \tfrac{1}{R^2}dx_1\wedge dx_2$, we see that $I_+$ coincides with $I_-$ along the curve $E_2=\{x_2=0\}$, and coincides with $-I_-$ along $E_1=\{x_1=0\}$. 
From the expression~\eqref{iplushopf}, we see that $\Omega_+$ spans an isotropic plane for the metric~\eqref{Hopfmetric}, hence $(g,I_+)$ is also Hermitian, with associated 2-form 
\begin{equation*}
\omega_+ = \tfrac{i}{4\pi R^2}(\theta_1\wedge \bar\theta_1 + \theta_2\wedge\bar\theta_2),
\end{equation*}
with $\theta_1 = \bar x_1 dx_1+x_2d\bar x_2$ and $\theta_2 = \bar x_1dx_2-x_2d\bar x_1$.  This 2-form also satisfies $d^c_+\omega_+ = H$, so that $(g,I_\pm)$ is an even generalized K\"ahler structure for $(\TT X, H)$.  From the explicit formulae~\eqref{reconstruct} for $\JJ_\pm$, we see that their real Poisson structures are given by 
\begin{equation}\label{Poissonexp}
Q_\pm = -\tfrac{1}{2}(\omega_+^{-1}\mp\omega^{-1}_-) = \tfrac{1}{2}(I_+\mp I_-)g^{-1}.
\end{equation}
Hence $Q_+$ drops rank from $4$ to $0$ along $E_2$, and $Q_-$ drops rank similarly on $E_1$. In other words, $\JJ_\pm$ are generically of symplectic type but each undergoes type change to complex type along one of the curves.   
\end{example}
The existence of generalized K\"ahler structures with nonzero torsion class on the Hopf surface implies, by Corollary~\ref{nonalg}, the well-known fact that the Hopf surface is non-algebraic.  It is natural to ask whether the Hopf surface might admit generalized K\"ahler structures with vanishing torsion class.  We now show that this is not the case.
\begin{prop}
Any generalized K\"ahler structure with $I_+$ given by the Hopf surface $X=(\CC^2-\{0\})/(x\mapsto 2x)$ must have nonvanishing torsion $[H]\in H^3(X,\RR)$. 
\end{prop}
\begin{proof}
The Hopf surface has $h^{2,1}=1$, generated by the $(2,1)$ component of the standard volume form $\nu$ of $S^3$, which satisfies $[\nu^{2,1}] = [\nu^{1,2}]=\tfrac{1}{2}[\nu]$ in de Rham cohomology.  Suppose that $X$ were the $I_+$ complex structure in a generalized K\"ahler structure with torsion $H$.  By the generalized K\"ahler condition, $H^{2,1} = -i\del_+\omega_+$, and so $d H^{2,1} = -i\delbar_+\del_+\omega_+=0$.  

We claim that $H^{2,1}$ must be nonzero in Dolbeault cohomology.  If not, we would have $\del_+\omega_+ = \delbar_+\tau$, for $\tau\in\Omega^{2,0}(X)$. Now let $E=\{x_1=0\}$, a null-homologous holomorphic curve in $X$, and let $D$ be a smooth 3-chain with $\del D  = E$. Then
\begin{equation*}
\int_E\omega_+ = \int_D d\omega_+ = \int_D d(\tau+\bar\tau) = \int_E(\tau+\bar\tau),
\end{equation*} 
which is a contradiction because $\omega_+$ is a positive $(1,1)$ form, forcing the left hand side to be nonzero, while $\tau$ is of type $(2,0)$ and vanishes on $E$.  

Because $h^{2,1}=1$, there must exist $\sigma\in\Omega^{2,0}$ such that $H^{2,1} = c\nu^{2,1} + \delbar\sigma$, with $c\in\CC^*$, and since $\delbar\sigma = d\sigma$, we have $[H^{2,1}] = c[\nu^{2,1}]$ in de Rham cohomology. But $[H^{2,1}]=\tfrac{1}{2}[H]$, since $H^{2,1}-H^{1,2} = -id\omega_+$.  Hence $[H]\neq 0$ in $H^3(M,\RR)$.   
\end{proof}
\begin{example}[Even-dimensional real Lie groups]
It has been known since the work of Samelson and
Wang~\cite{Samelson}, \cite{Wang} that any 
even-dimensional real Lie group $G$ admits left- and right-invariant
complex structures $J_L,J_R$.  If the group admits a bi-invariant positive-definite inner product $b(\cdot,\cdot)$, the complex structures can be chosen to be Hermitian with respect to $b$.  The bi-Hermitian structure $(b,J_L,J_R)$ is then a generalized K\"ahler structure on $(\TT G, H)$, where $H$ is the Cartan 3-form associated to $b$, defined by $H(X,Y,Z)=b([X,Y],Z)$. To see this, we compute $d^c_{J_L}\omega_{J_L}$:
\begin{equation*}
\begin{split}
A=d^c_{J_L}\omega_{J_L}(X,Y,Z)&=d\omega_{J_L}(J_LX,J_LY,J_LZ)\\
&=-b([J_LX,J_LY],Z)+c.p.\\
&=-b(J_L[J_LX,Y]+J_L[X,J_LY]+[X,Y],Z)+ c.p.\\
&=(2b([J_LX,J_LY],Z)+c.p.) - 3H(X,Y,Z)\\
&=-2A - 3H(X,Y,Z),
\end{split}
\end{equation*}
Proving that $d^c_{J_L}\omega_{J_L}=-H$.  Since the right Lie
algebra is anti-isomorphic to the left, the same calculation with
$J_R$ yields $d^c_{J_R}\omega_{J_R}=H$, and finally we have
\begin{equation*}
-d^c_{J_L}\omega_{J_L}=d^c_{J_R}\omega_{J_R}=H,
\end{equation*}
as required.  For $G=SU(2)\times U(1)$, we recover Example~\ref{evenhopf} from this construction.  Note that $J_L,J_R$ are isomorphic as complex manifolds, via the inversion on the group.
\end{example}

\section{Holomorphic Dirac geometry}\label{Holomorphicdiracgeometry}

In the previous section, we saw that a generalized K\"ahler structure on $M$ gives rise to a pair of complex manifolds $X_\pm = (M,I_\pm)$ with the same underlying smooth manifold.  In this section, we apply the results of Section~\ref{gerbdir} to Courant algebroids (and their prequantum gerbes) carrying generalized K\"ahler structures, shedding light on the relationship between the complex manifolds $X_\pm$.  Although there is no morphism between $X_+$ and $X_-$ in the holomorphic category, we show that $X_\pm$ are each equipped with holomorphic Courant algebroids which decompose as a sum of transverse holomorphic Dirac structures, and that the Dirac structures on $X_+$ are Morita equivalent to those on $X_-$.  This provides a holomorphic interpretation to the deformation theory of generalized complex structures, as well as to the notion of a generalized holomorphic bundle.   

\subsection{Holomorphic reduction}\label{holreduct}

Our main tool will be the decomposition~\eqref{kahldeco} of the Courant algebroid induced by the generalized K\"ahler structure:
\begin{equation*}
E\otimes\CC  =  \ell_+\oplus \ell_-\oplus\overline\ell_+\oplus\overline\ell_-.
\end{equation*}
The bundles $\overline\ell_\pm$ satisfy Definition~\ref{deflift}, since they are involutive isotropic liftings of the antiholomorphic tangent bundles of the two complex structures $I_\pm$.  Theorem~\ref{courvers} immediately yields the following.  
\begin{prop}\label{holcourants}
The bundles $\overline \ell_\pm$ are liftings of $T_{0,1} X_\pm$ to $E\otimes\CC$, hence define two reductions of $E\otimes\CC$ to holomorphic Courant algebroids $\EE_\pm$ over the complex manifolds $X_\pm$.  
\begin{equation}\label{fm}
\xymatrix{ & (E\otimes\CC,M)\ar@{-->}[dl]_{\overline\ell_-}\ar@{-->}[dr]^{\overline\ell_+} & \\ (\EE_-,X_-) & & (\EE_+,X_+)}
\end{equation}
Furthermore, the isomorphism class $[\EE_\pm]$ is given by the $(2,1)$ component (with respect to $I_\pm$) of the torsion of the generalized K\"ahler metric:
\begin{equation*}
[\EE_\pm] = [2H^{(2,1)_\pm}]\in H^1(X_\pm;\Omega^{2,\cl}).
\end{equation*}
\end{prop}
\begin{proof}
The existence of the reductions follows from Theorem~\ref{courvers}, as explained above.  To compute the isomorphism classes, we write the lifting $\overline\ell_\pm$ explicitly using Equation~\eqref{ellexp2}, namely $\ell_\pm = e^{\mp i\omega_\pm}s(T_{1,0}X_\pm)$, and use the explicit form for the cocycle given in Remark~\ref{explclass}, yielding
\begin{equation*}
[\EE_\pm]=[H^{(2,1)_{\pm}}  + \del (\mp i\omega_\pm)].
\end{equation*}
Since $H^{(2,1)_\pm} = \mp i \del\omega_\pm$ from~\eqref{constraint}, we obtain the required cocycle $2H^{(2,1)_\pm}$. 
\end{proof}
\begin{remark}\label{pmsplit}
To obtain a more explicit expression for $\EE_\pm$, we may use the canonical splitting $s$ given by the generalized K\"ahler metric to (smoothly) split the sequence
\begin{equation*}
\xymatrix{0\ar[r] & T^*_{1,0}X_\pm\ar[r] & \EE_\pm\ar[r] & T_{1,0}X_\pm\ar[r] & 0},
\end{equation*}
by defining the following map $s_\pm:T_{1,0}X_\pm\lra \EE_\pm = \overline\ell_\pm^\bot/\overline\ell_\pm$:
\begin{equation*}
s_\pm(X):= s(X)\mp gX =s(X)\pm i\omega_\pm X\mod \overline\ell_\pm,
\end{equation*}
 for $X\in T_{1,0}X_\pm$.  The holomorphic structure on $\EE_\pm$ is then computed via~\eqref{courdelbar}, using the Courant bracket on $\TT M$ given by the torsion 3-form $H$. The resulting Courant algebroid is $\EE_\pm = T_{1,0}X_\pm\oplus T^*_{1,0}X_\pm$, with a modified holomorphic structure as in Example~\ref{genex}:
\begin{equation*}
\overline{D}_\pm = \begin{pmatrix}\delbar_\pm& 0\\2H^{(2,1)_\pm}&\delbar_\pm\end{pmatrix}.
\end{equation*}
\end{remark}

The holomorphic Courant algebroids $(\EE_\pm,X_\pm)$ can be very different, with non-isomorphic underlying complex manifolds $X_\pm$. Nevertheless, they are closely related, as they are both reductions of one and the same smooth Courant algebroid, where the liftings $\overline\ell_\pm$ defining them are compatible in the sense that $\overline\ell_+ \oplus \overline\ell_-\subset E\otimes\CC$ is a Dirac structure, and therefore itself a Lie algebroid.  This configuration is well-known in the literature and is called a \emph{matched pair} of Lie algebroids~\cite{MR1430434, MR1460632}.  We now describe several consequences of these two compatible reductions.

Following the philosophy of symplectic reduction applied to Courant algebroids, the reduction of $E$ to $\EE_\pm$ allows us to reduce Dirac structures from $E$ to $\EE_\pm$.
As in Theorem~\ref{courvers}, we write the reduction of $E$ by a lifting $D\subset E$ as $E_D = D^\bot / D$.  Then the reduction of Dirac structures proceeds as follows~\cite{MR2323543}.  
\begin{prop}
Let $L\subset E$ be a Dirac structure such that $L\cap D^\bot$ has constant rank and $L$ is $D$-invariant, in the sense $[\OO(D),\OO(L)]\subset\OO(L)$. Then the subbundle $L_D\subset E_D$, defined by 
\begin{equation}\label{diracred}
L_D  :=  \frac{L\cap D^\bot + D}{D},
\end{equation}
is holomorphic with respect to the induced holomorphic structure~\eqref{courdelbar} on $E_D$ and defines a holomorphic Dirac structure in $E_D$ called the reduction of $L$.
\end{prop}
Using this Dirac reduction, we show that the Dirac structures in $E$ corresponding to the $\pm i$-eigenbundles of $\JJ_\pm$ descend to holomorphic Dirac structures in $\EE_\pm$.
\begin{theorem}\label{maintransv}
Each of the holomorphic Courant algebroids $\EE_\pm$ over the complex manifolds underlying a generalized K\"ahler manifold contains a pair of transverse holomorphic Dirac structures
\begin{equation*}
\EE_\pm = \Aa_\pm\oplus\Bb_\pm,
\end{equation*}
where $\Aa_\pm$ are both the reduction of the $-i$-eigenbundle of $\JJ_+$ and $(\Bb_+,\Bb_-)$ are reductions of the $-i$ and $+i$ eigenbundles of $\JJ_-$, respectively.
\end{theorem}
\begin{proof}
Consider the reduction by $\overline\ell_-$. The Dirac structures in $E\otimes\CC$ given by the $-i$-eigenbundle of $\JJ_+$ and the $+i$-eigenbundle of $\JJ_-$ are as follows:
\begin{equation*}
\overline{L}_+ =\overline\ell_+\oplus\overline\ell_-,\hspace{50pt} 
L_-=\ell_+\oplus\overline\ell_-.
\end{equation*}
Since they both contain $\overline\ell_-$ as an involutive subbundle, it follows that both Dirac structures have intersection with $\overline\ell_-^\bot$ of constant rank, and also that both $L_-,\overline{L}_+$ are $\overline\ell_-$-invariant.  Hence by~\eqref{diracred}, they reduce to holomorphic Dirac structures $\Aa_-, \Bb_-$ in the holomorphic Courant algebroid $\EE_-$.  These are transverse simply because $\ell_+,\overline\ell_+$ have zero intersection.   The same argument applies for the reduction by $\overline\ell_+$. 
\end{proof}
\begin{remark}\label{explalg}
Using the splittings $s_\pm$ from Remark~\ref{pmsplit}, we can describe the Dirac structures explicitly as follows.  For simplicity, we describe the Dirac structure $\Bb_+$ inside $\EE_+  = \Aa_+\oplus \Bb_+$.  The Dirac structure $\Bb_+$ is obtained by reduction of $\bar\ell_+\oplus\ell_-$, which has image in $\overline\ell_+^\bot/\overline\ell_+$ isomorphic to $\ell_-\isom T_{1,0}X_-$.  Hence we give a map $T_{1,0}X_-\lra\EE_+$ with image $\Bb_+$. 

Let $P_+$ be the projection of a vector to $T_{1,0}X_+$, and let $\bar P_+$ be the complex conjugate projection.  For any $X\in T_{1,0}X_-$, we have $X-gX\in\ell_-$, and therefore
\begin{align*}
X-gX&=(P_+X+\overline P_+ X) - g(P_+X+\overline P_+X)\\
&=(P_+X-gP_+X) +(\overline P_+X-g\overline P_+X)\\
&=(P_+X-gP_+X) -2g\overline P_+X\ \ (\text{mod}\ \overline\ell_+),
\end{align*}
where the last two terms are in $s_+(T_{1,0}X_+)$ and $T^*_{1,0}X_+$, respectively. Hence the map \begin{equation*}
X\mapsto P_+X -2g\overline P_+X
\end{equation*}
 is an isomorphism $T_{1,0}X_-\lra \Bb_+$.  In fact the same map gives an isomorphism $T_{0,1} X_-\lra \Aa_+$.  
\end{remark}
\begin{remark}As complex bundles, $\Bb_\pm$ are isomorphic to $T_{(1,0)}X_\mp$.  In other words, $\Bb_+$, a holomorphic Dirac structure on $X_+$, is isomorphic as a smooth bundle to the holomorphic tangent bundle of the opposite complex manifold $X_-$, and vice versa for $\Bb_-$. The way in which the holomorphic tangent bundle of $X_-$ acquires a holomorphic structure with respect to $X_+$ seems particularly relevant to the study of so-called heterotic compactifications with $(2,0)$ supersymmetry~\cite{MR851702, MR2192064}, where only one of the complex structures $I_\pm$ is present, but there is an auxiliary holomorphic bundle which appears to play a similar role to $\Bb_\pm$.   
\end{remark}
The presence of transverse Dirac structures in each of $\EE_\pm$ immediately implies, by~\eqref{transdirac} and the surrounding discussion, that $X_\pm$ inherit holomorphic Poisson structures.  We now describe these explicitly, and verify that they coincide with the holomorphic Poisson structures on generalized K\"ahler manifolds found in~\cite{MR2217300}.  
\begin{prop}\label{holpoiss}
By forming the Baer sum $\Aa_\pm^\top\boxtimes\Bb_\pm$, the transverse Dirac structures $\Aa_\pm,\Bb_\pm$ give rise to holomorphic Poisson structures $\sigma_\pm$ on the complex manifolds $X_\pm$, both of which have real part 
\begin{equation*}
\mathrm{Re}(\sigma_\pm)=\tfrac{1}{8}g^{-1}[I_+^*,I_-^*].
\end{equation*}
\end{prop}
\begin{proof}
Following~\eqref{poistr}, we compute $\sigma_\pm$ explicitly using the decomposition $\overline\ell_\pm^\bot = \ell_\mp\oplus\overline\ell_+\oplus\overline\ell_-$ and the canonical splitting of $E$ given by the generalized  K\"ahler metric.  Let $P_{\pm} = \tfrac{1}{2}(1-iI_\pm)$ be the projection of a vector or covector to its $(1,0)_\pm$ part, and let $\overline P_\pm$ be its complex conjugate.  Then $\sigma_\pm$ applied to $\xi\in T^*X_\pm\otimes\CC$ is given by taking the component $\alpha$ of $P_\pm\xi\in \overline\ell_\pm^\bot$ along $\ell_\mp$, and then projecting it to $(TX_\pm\otimes\CC)/T_{0,1}X_\pm$. Computing $\alpha$, we obtain
\begin{equation*}
(\overline P_\mp P_\pm\xi) \mp g^{-1}(\overline P_\mp P_\pm\xi),
\end{equation*}
and projecting $\alpha$ we obtain $\mp P_\pm g^{-1}\bar P_\mp P_\pm\xi$, so that our expression for $\sigma_\pm$ is 
\begin{align}
\sigma_\pm &=  \mp g^{-1}\overline P_\pm\overline P_\mp P_\pm\label{formpois}\\
&=\mp\tfrac{1}{8}g^{-1}(1\pm i I^*_\pm)(1\pm iI^*_\mp)(1\mp iI^*_\pm)\notag\\
&=\tfrac{1}{8}g^{-1}([I^*_+,I^*_-] +i I^*_\pm[I^*_+,I^*_-]).\notag
\end{align}
\end{proof}
\begin{example}
To give a description of the $(\Aa_-,\Bb_-)$ Dirac structures for the generalized K\"ahler structure on the Hopf surface from Example~\ref{evenhopf}, we compute the isomorphism $T_{1,0}X_+\lra\Aa_-$ as in Remark~\ref{explalg}, yielding $X\mapsto  P_-X + 2g\overline{P}_-X$, and apply it to the basis of $(1,0)_+$ vectors given by $R^{-2}g^{-1}(\bar x_1 dx_1+x_2d\bar x_2)$ and $R^{-2}g^{-1}(\bar x_1 dx_2-x_2d\bar x_1)$, obtaining the following basis of holomorphic sections for $\Aa_-$:
\begin{equation*}
x_2\tfrac{\del}{\del x_2} + \tfrac{1}{2\pi R^2} \bar x_1 dx_1,\ \ \ \ \ -x_2\tfrac{\del}{\del x_1} + \tfrac{1}{2\pi R^2} \bar x_1 dx_2.
\end{equation*}
The same prescription produces a basis for $\Bb_-$:
\begin{equation*}
x_1\tfrac{\del}{\del x_1} + \tfrac{1}{2\pi R^2} \bar x_2 dx_2,\ \ \ \ \ x_1\tfrac{\del}{\del x_2} - \tfrac{1}{2\pi R^2} \bar x_2 dx_1.
\end{equation*}
We see from this that the anchor map for $\Aa_-$ is an isomorphism except along the curve $\{x_2=0\}$ where it has rank zero, whereas the anchor map for $\Bb_-$ drops rank along $\{x_1=0\}$.
Computing the Poisson tensor $\sigma_-$ using~\eqref{formpois} yields
\begin{equation*}
\sigma_-  = -x_1x_2\tfrac{\del}{\del x_1}\tfrac{\del}{\del x_2},
\end{equation*}
which is an anticanonical section vanishing on the union of the degeneration loci of $\Aa_-$ and $\Bb_-$.
\end{example}

The fact that the $\pm i$ eigenbundles of $\JJ_\pm$ descend to transverse holomorphic Dirac structures provides a great deal of information concerning the classical geometry that they determine on the base manifold.  Just as in Section~\ref{gcdg}, where we discussed the transverse singular foliations induced on a manifold by transverse Dirac structures, we have a similar result here.
\begin{prop}
The transverse holomorphic Dirac structures $\Aa_\pm,\Bb_\pm$ in $\EE_\pm$ induce transverse holomorphic singular foliations $\FF_\pm, \GG_\pm$ on $X_\pm$. The intersection of a leaf of $\FF_\pm$ with a leaf of $\GG_\pm$ is a (possibly disconnected) symplectic leaf for the holomorphic Poisson structure $\sigma_\pm$.  
Furthermore, the singular foliations $\FF_\pm,\GG_\pm$ coincide with the foliations induced by the generalized complex structures $\JJ_+, \JJ_-$, respectively.
\end{prop}
\begin{proof}
The behaviour of the holomorphic Dirac structures is precisely as in the real case, discussed in Section~\ref{gcdg}.  To see why the holomorphic foliations coincide with the generalized complex foliations, we may appeal to the reduction procedure for Dirac structures, which makes it evident.  

Alternatively, observe that in order to extract a foliation from a holomorphic Lie algebroid $\Aa$ over $X$, one possible way to proceed is to first represent the holomorphic Lie algebroid as a smooth Lie algebroid with compatible holomorphic structure, by forming the associated complex Lie algebroid $A=\Aa\oplus T_{0,1}X$ as in~\cite{MR2439547}.  Then take the fiber product over $TX\otimes\CC$ with the complex conjugate $\overline{A}=\overline{\Aa}\oplus T_{1,0}X$, to obtain a real Lie algebroid, defining a foliation of $X$. 

Applying this to the Lie algebroid $\Aa_\pm$ over $X_\pm$, we see immediately that the associated complex Lie algebroid $\Aa_\pm\oplus T_{0,1}X_\pm$ is precisely $\overline \ell_+\oplus\overline\ell_-=\overline{L}_+$, the $-i$-eigenbundle of $\JJ_+$. Furthermore, the fiber product construction yields
\begin{equation*}
\Aa_\pm'\otimes_{TX_\pm\otimes\CC} \overline{\Aa}'_\pm = \overline{L_+}\otimes_{TX_\pm\otimes\CC}L_+,
\end{equation*}
which by~\eqref{relpoi} is the Lie algebroid corresponding to the real Poisson structure associated to $\JJ_+$.  The same argument applies to $\Bb_\pm$, relating its holomorphic foliation to that determined by $\JJ_-$.
\end{proof}
\begin{remark}
According to the above proposition, the generalized foliation induced by a generalized complex structure $\JJ_\pm$ in a generalized K\"ahler pair is holomorphic with respect to $I_+$ and $I_-$. The relationship between the symplectic structure of the leaves and the induced complex structures from $I_\pm$ may be understood by applying the theory of generalized K\"ahler reduction, as follows.  

Let $S\subset M$ be a submanifold and $K=N^*S$ its conormal bundle.   We saw in~\eqref{redsub} that a Courant algebroid $E$ on $M$ may be pulled back to $S$ to yield a Courant algebroid $E_S = K^\bot/K$ over $S$.  If a generalized complex structure $\JJ$ on $E$ satisfies $\JJ K \subset K$, then it induces a generalized complex structure $\JJ_{\mathrm{red}}$ on the reduced Courant algebroid $E_S$.  If $S$ is a leaf of the real Poisson structure $Q$ associated to $\JJ$, it follows that $\JJ K \subset K$ and that $\JJ_{\mathrm{red}}$ is of symplectic type, reproducing the symplectic structure derived from $Q$.  
 
In~\cite{MR2323543, MR2397619}, it is shown that if $(\JJ_+,\JJ_-)$ is a generalized K\"ahler structure for which $\JJ_+K=K$ as above, then the entire generalized K\"ahler structure reduces to $E_S$, with $(\JJ_+)_{\mathrm{red}}$ of symplectic type.  In particular, we obtain a bi-Hermitian structure on $S$.  We may then perform a second generalized K\"ahler reduction, from $S$ to a symplectic leaf of $(\JJ_-)_{\mathrm{red}}$, whereupon we obtain a generalized K\"ahler structure where both generalized complex structures are of symplectic type. 
\end{remark}

\subsection{Sheaves of differential graded Lie algebras}
A Dirac structure $\Aa\subset \EE$ is, in particular, a Lie algebroid, so that the de Rham complex $(\Omega^\bullet_\Aa, d_\Aa)$ is a sheaf of differential graded algebras.  If $(\Aa,\Bb)$ is a pair of transverse Dirac structures, then as was observed in~\cite{MR1472888}, the de Rham complex inherits further structure.  It is shown there that if we make the identification $\Bb =\Aa^*$ using the symmetric pairing on $\EE$, then the Lie bracket on $\Bb$ extends to a differential graded Lie algebra structure on $\Omega^\bullet_\Aa$, so that
\begin{equation*}
(\Omega^\bullet_\Aa, d_\Aa, [\cdot,\cdot]_{\Bb}) 
\end{equation*}
is a sheaf of differential graded Lie algebras (the degree is shifted so that $\Omega^k_\Aa$ has degree $k-1$).  

Given a differential graded Lie algebra as above, there is a natural question which arises: what is the object whose deformation theory it controls?  In~\cite{MR1472888}, the above differential graded Lie algebra was explored in the smooth category, in which case there is a direct interpretation in terms of deformations of Dirac structures.  A deformation of the Dirac structure $\Aa$ inside $\EE = \Aa\oplus\Bb$ may be described as the graph of a section $\eps\in\Omega^2_\Aa(M)$, viewed as a map $\eps:\Aa\lra \Bb=\Aa^*$.  It is shown in~\cite{MR1472888} that the involutivity of this graph is equivalent to the Maurer-Cartan equation:
\begin{equation}\label{MC}
d_\Aa \eps + \tfrac{1}{2}[\eps,\eps]_\Bb = 0.
\end{equation}
This leads, assuming that $(\Omega^\bullet_\Aa,d_\Aa)$ is an elliptic complex and $M$ is compact, to a finite-dimensional moduli space of deformations of $\Aa$ in $\EE$, presented as the zero set of an obstruction map $H^2_{d_\Aa}\lra H^3_{d_A}$.  

The deformation theory governed by a sheaf of differential graded Lie algebras in the holomorphic category is much more subtle, for the reason that the objects being deformed are not required to be given by global sections of the sheaf (of which there may be none).  The objects are considered to be ``derived'' in the sense that the Maurer-Cartan equation~\eqref{MC} is not applied to global sections in $\Omega^2_{\Aa}(X)$ but rather to the global sections in total degree 2 of a resolution $\II^{\bullet\bullet}$ of the complex $\Omega^\bullet_\Aa$. Note that the structure of the resolution $\II^{\bullet\bullet}$ may not, in general, be that of a differential graded Lie algebra, but only one up to homotopy, so one must interpret the Maurer-Cartan equation appropriately.  In any case, the moduli space is then given by an obstruction map between the derived global sections of the differential complex $(\Omega^\bullet_\Aa, d_\Aa)$, namely the hypercohomology groups.  In short, we expect a moduli space described by an obstruction map 
\begin{equation*}
\HH^2(\Omega^\bullet_{\Aa}, d_\Aa)\lra \HH^3(\Omega^\bullet_{\Aa}, d_\Aa).
\end{equation*}
General results concerning such deformation theories can be found, for example, in~\cite{MR1439623, MR2344349}, and a case relevant to generalized geometry has been investigated in~\cite{Chen:2009xs}.  

We wish simply to observe that in our case, since the holomorphic Dirac structures $(\Aa_\pm,\Bb_\pm)$ are obtained by reduction from smooth Dirac structures in $E\otimes\CC$, their de Rham complexes are equipped with canonical resolutions by fine sheaves, which are themselves differential graded Lie algebras controlling a known deformation problem.  We conclude with the main result of this section, which may be viewed as a holomorphic description for the deformation theory of generalized complex structures, under the assumption of the generalized K\"ahler condition.
\begin{prop}
The derived deformation complex defined by the sheaf of holomorphic differential graded algebras $(\Omega^\bullet_{\Aa_+}, d_{\Aa_+}, [\cdot,\cdot]_{\Bb_+})$ on the complex manifold $X_+$ is canonically isomorphic to that defined by $(\Omega^\bullet_{\Aa_-}, d_{\Aa_-}, [\cdot,\cdot]_{\Bb_-})$ on the complex manifold $X_-$: they both yield the deformation complex of the generalized complex structure $\JJ_+$. 

Similarly, the sheaves of differential graded Lie algebras $(\Omega^\bullet_{\Bb_\pm}, d_{\Bb_\pm}, [\cdot,\cdot]_{\Aa_\pm})$ have derived deformation complexes which are canonically complex conjugate to each other, and are naturally isomorphic to the deformation complex of the generalized complex structure $\JJ_-$.  
\end{prop} 
\begin{proof}
Consider the $-i$ eigenbundle of $\JJ_+$, given by $\overline{L}_+ = \overline\ell_+\oplus\overline\ell_-$.  Because $\overline{L}_+$ decomposes into the involutive Lie sub-algebroids $\overline\ell_\pm$, its de Rham complex is the total complex of a double complex:
\begin{equation}\label{resdgla}
\Omega^k_{\overline{L}_+}= \bigoplus_{p+q=k}\OO(\wedge^p\overline\ell^*_-\otimes\wedge^q\overline\ell^*_+),\hspace{30pt} d_{\overline{L}_+} = d_{\overline\ell_-} + d_{\overline\ell_+}.
\end{equation}
Identifying $\ell_\pm^*$ with $\overline\ell_\pm$ using the symmetric pairing on $E$, the above double complex inherits a Lie bracket from the Lie algebroid $L_+ = \ell_-\oplus\ell_+$.  Furthermore, since $(L_+,\overline{L}_+)$ forms a Lie bialgebroid, we obtain that the Lie bracket on~\eqref{resdgla} is compatible with the bi-grading and the differential.  Finally, recall that $\overline\ell_\pm$ is isomorphic to $T_{(0,1)}X_\pm$.   As a result, we may view the differential $\ZZ\times\ZZ$-graded Lie algebra~\eqref{resdgla} in two ways:
\begin{enumerate}
\item Horizontally, using the differential $d_{\overline\ell_-}$, the complex is a Dolbeault resolution, over the complex manifold $X_-$, of the de Rham complex of the holomorphic Lie algebroid $\Aa_-$.  The inclusion of $\Omega^\bullet_{\Aa_-}$ in the double complex is also a homomorphism of Lie algebras.
\item Vertically, using the differential $d_{\overline\ell_+}$, the complex is a Dolbeault resolution, over $X_+$, of the de Rham complex of $\Aa_+$. Also, the inclusion of $\Omega^\bullet_{\Aa_+}$ is  a homomorphism of Lie algebras.
\end{enumerate}
On the other hand, the total complex of this double complex has already been interpreted; as we saw in Section~\ref{gcdg}, the differential graded Lie algebra $(\Omega^\bullet_{\overline{L}_+}, d_{\overline{L}_+}, [\cdot,\cdot]_{L_+})$ controls the deformation theory of the generalized complex structure $\JJ_+$.  The statement for $(\Omega^\bullet_{\Bb_\pm}, d_{\Bb_\pm}, [\cdot,\cdot]_{\Aa_\pm})$ is shown in the same way, using instead the $\pm i$-eigenbundles of $\JJ_-$.
\end{proof}
In particular, the above result implies the following fact, striking from the point of view of the complex manifolds $X_\pm$, which are not related in any obvious holomorphic fashion:
\begin{corollary}
We have the following canonical isomorphisms of hypercohomology for the de Rham complexes of the holomorphic Dirac structures $(\Aa_-,\Bb_-)$ on $X_-$ and $(\Aa_+,\Bb_+)$ on $X_+$:
\begin{align*}
\HH^k(X_-,\Omega^\bullet_{\Aa_-}) &= \HH^k(X_+,\Omega^\bullet_{\Aa_+});\\
\HH^k(X_-,\Omega^\bullet_{\Bb_-}) &= \overline{\HH^k(X_+,\Omega^\bullet_{\Bb_+})}. 
\end{align*}
\end{corollary}

\subsection{Morita equivalence}

In the previous section, we saw that the pair $(\Aa_+, \Bb_+)$ of transverse holomorphic Dirac structures on the complex manifold $X_+$ is closely related to its counterpart $(\Aa_-,\Bb_-)$ on $X_-$, in that the Dirac structures $\Aa_\pm$ have identical derived deformation theory and hypercohomology groups, and similarly for $\Bb_\pm$.  The purpose of this section is to describe the relationship between $(X_+,\Aa_+,\Bb_+)$ and $(X_-,\Aa_-,\Bb_-)$ as a Morita equivalence.  Morita equivalence for Lie algebroids in the smooth category is well-studied in Poisson geometry~\cite{MR1138048, MR1959580, MR2002661} and the version we develop here is a special case, with additional refinements made possible by the complex structures which are present.  We use the result from~\cite{MR2439547} that a holomorphic Lie algebroid $\LL$ on $X$ may be described equivalently by a complex Lie algebroid structure on $L = \LL\oplus T_{0,1}X$, compatible with the given holomorphic data. 

\begin{defn}
Let $\varphi_\pm:M\lra X_\pm$ be diffeomorphisms\footnote{We take $\varphi_\pm$ to be diffeomorphisms purely for convenience.} from a manifold $M$ to two complex manifolds $X_\pm$, and let $\LL_\pm$ be holomorphic Lie algebroids on $X_\pm$.  Then $\LL_+$ is Morita equivalent to $\LL_-$ when there is an isomorphism $\psi$ between the associated complex Lie algebroids $L_\pm:=\LL_\pm\oplus T_{0,1}X_\pm$, i.e.:
\begin{equation*}
\xymatrix{\varphi_+^*L_+\ar[rr]^\psi\ar[dr] & & \varphi_-^*L_-\ar[dl]\\ & M &}
\end{equation*}
Similarly, $\LL_+$ is Morita conjugate to $\LL_-$ when there is an isomorphism of complex Lie algebroids from $L_+$ to $\overline L_-$. 
\end{defn}

\begin{prop}
Let $(\Aa_\pm,\Bb_\pm)$ be the transverse holomorphic Dirac structures on the complex manifolds $X_\pm$ participating in a generalized K\"ahler structure. Then $\Aa_+$ is Morita equivalent to $\Aa_-$, and $\Bb_+$ is Morita conjugate to $\Bb_-$. 
\end{prop}
\begin{proof}
This is an immediate consequence of the canonical isomorphisms of complex Lie algebroids
\begin{align*}
A_+ = \Aa_+\oplus T_{0,1}X_+ = \overline\ell_-\oplus\overline\ell_+ = \Aa_-\oplus T_{0,1}X_- = A_-\\
B_+ = \Bb_+\oplus T_{0,1}X_+ = \ell_- \oplus \overline\ell_+ = \overline{\Bb_-}\oplus T_{1,0}X_- = \overline B_-.
\end{align*}
\end{proof}
Just as in~\cite{MR1959580}, the Morita equivalence between $\Aa_+,\Aa_-$ induces an equivalence between their $\CC$-linear categories of modules.  Similarly the Morita conjugacy between $\Bb_+,\Bb_-$ implies a $\CC$-antilinear equivalence of their module categories. Since the Morita equivalence is an isomorphism at the level of complex Lie algebroids over $M$, we can also strengthen this statement to an equivalence of the DG categories of cohesive modules~\cite{Block:jt}, i.e. representations up to homotopy~\cite{Abad:2009ls}.  We only remark here that these modules have a generalized complex interpretation, since $A_\pm = \overline\ell_+\oplus\overline\ell_-$ is the $-i$-eigenbundle of $\JJ_+$, whose modules are, by definition, generalized holomorphic bundles~\cite{Gualtieri:qr}.
\begin{corollary}
The categories of holomorphic $\Aa_\pm$-modules are equivalent to each other and to the category of generalized holomorphic bundles for $\JJ_+$.  Similarly, the category of holomorphic $\Bb_+$-modules is equivalent to the category of generalized holomorphic bundles for $\JJ_-$, and $\CC$-antilinearly equivalent to the category of modules for $\Bb_-$.
\end{corollary}

A special case occurs when $\JJ_-$ is of symplectic type; in this case $\Bb_\pm$ are isomorphic as holomorphic Lie algebroids with $T_{1,0}X_\pm$.  But $T_{1,0}X_+$ is a holomorphic Lie algebroid which is actually Morita conjugate to itself, via the complex conjugation map 
\begin{equation*}
T_{1,0}X_+\oplus T_{0,1}X_+\stackrel{c.c.}{\lra}T_{1,0}X_+\oplus T_{0,1}X_+.
\end{equation*}
Hence, composing this with the Morita conjugacy $\Bb_+\lra \Bb_-$, we obtain that $\Bb_+$ is Morita equivalent to $\Bb_-$. This is significant because we then have a Morita equivalence between the fiber product of the Lie algebroids over $T_{1,0}X_\pm$:
\begin{equation*}
\Aa_+\oplus_{T_{1,0} X_+}\Bb_+\lra \Aa_-\oplus_{T_{1,0} X_-}\Bb_-.
\end{equation*}  
But by Proposition~\ref{holpoiss}, these fiber products are the holomorphic Lie algebroids corresponding to the holomorphic Poisson structures $\sigma_\pm$ on $X_\pm$.  Hence we obtain a Morita equivalence between holomorphic Poisson structures, generalizing the result in~\cite{Gualtieri:2007fe} on Morita equivalence for a specific construction of generalized K\"ahler structures.
\begin{corollary}
If either $\JJ_+$ or $\JJ_-$ is of symplectic type, then the holomorphic Poisson structures $\sigma_\pm$ on $X_\pm$ are Morita equivalent. 
\end{corollary}
\subsection{Prequantization and holomorphic gerbes}

Geometric quantization of symplectic manifolds is perhaps best understood in the setting of K\"ahler geometry.  A symplectic manifold $(M,\omega)$ is said to be prequantizable when $[\omega]/2\pi\in H^2(M,\RR)$ has integral periods, i.e. is in the image of the natural map $H^2(M,\ZZ)\lra H^2(M,\RR)$.  A prequantization of such an integral symplectic form is a Hermitian complex line bundle $(L,h)$ equipped with a unitary connection $\nabla$ such that $F(\nabla) = i\omega$.  The presence of a complex structure $I$ compatible with $\omega$, sometimes called a complex polarization, then implies that $\nabla^{0,1}$ defines a holomorphic structure on the line bundle $L$, which is used to proceed with the geometric quantization procedure.  In this sense, we view a Hermitian holomorphic line bundle over a complex manifold $(M,I,L,h)$ as a prequantization of the K\"ahler structure $(M, I,\omega)$.  In this section, we seek an analogous result for generalized K\"ahler manifolds.

Our first task is to prequantize the underlying Courant algebroid $E$.  For this to be possible, we need the quantization condition that $[E]/2\pi\in H^3(M,\RR)$ has integral periods, and choose a Hermitian gerbe $(G,h)$ with unitary $0$-connection $\nabla$ such that the associated Courant algebroid $E_\nabla$ (via Corollary~\ref{unitgerb}) satisfies $[E_\nabla] = [E]$.  This is always possible since the map~\eqref{surjonint} is surjective onto classes vanishing in $H^3(M,\RR/\ZZ)$. 

If $E_\nabla$ carries a generalized complex structure $\JJ$, it immediately obtains Dirac structures $L,\overline{L}$ given by $\ker(\JJ\mp i)$, and by Theorem~\ref{diracflat}, this induces flat connections on $G$ over these Lie algebroids.  By analogy with vector bundles, we say that a gerbe with a flat $\overline{L}$-connection is a generalized holomorphic gerbe.
\begin{prop}\label{gcgerbe}
Let $\JJ$ be a generalized complex structure on $E_\nabla$, where $\nabla$ is a unitary $0$-connection on the Hermitian gerbe $(G,h)$. Then $G$ inherits a generalized holomorphic structure.  Furthermore, it has a flat Poisson connection for the underlying real Poisson structure. 
\end{prop}
\begin{proof}
$G$ inherits a flat $\overline{L}$-connection by Theorem~\ref{diracflat}.  To show that it has a flat Poisson connection, we note that the trivial gerbe has a canonical flat $L$-connection, and by tensoring with $G$ we obtain a flat connection on $G$, for the fiber product of $\overline{L}$ with $L$, which by~\eqref{relpoi} is the Lie algebroid of the Poisson structure $Q$ underlying $\JJ$, as required.  
\end{proof}

Applying the above result to each generalized complex structure separately, a generalized K\"ahler structure $(\JJ_+,\JJ_-)$ on $E_\nabla$ induces flat  $\overline{L}_\pm$-connections on $G$, rendering it generalized holomorphic with respect to both $\JJ_\pm$. 
\begin{corollary}\label{gholgerb}
Let $(\JJ_+,\JJ_-)$ be a generalized K\"ahler structure on $E_\nabla$, which is as above.  Then the gerbe $G$ inherits  generalized holomorphic structures over $\JJ_+$ and $\JJ_-$. In particular, it obtains flat Poisson connections over the real Poisson structures $Q_\pm$ underlying $\JJ_\pm$.
\end{corollary}

We may interpret the above generalized holomorphic structures in terms of the underlying bi-Hermitian geometry, as follows.
\begin{prop}
Let $(G,h,\nabla,\JJ_+,\JJ_-)$ be as above. Then $G$ inherits holomorphic structures with respect to the underlying complex manifolds $X_\pm$, defining holomorphic gerbes $\GG_\pm$.  Furthermore, $\GG_\pm$ inherit holomorphic $0$-connections $\del_\pm$, whose associated holomorphic Courant algebroids $\EE_\pm$ are given in Proposition~\ref{holcourants}. 
\end{prop}
\begin{proof}
The commuting of $\JJ_\pm$ induces the decomposition~\eqref{defell}, and since $\overline{\ell}_\pm$ are liftings of $T_{0,1}X_\pm$, Theorem~\ref{holconn} implies that the gerbe $G$ inherits the structure of a holomorphic gerbe with holomorphic $0$-connection over $X_\pm$, with associated holomorphic Courant algebroid $\EE_\pm$ given by the reduction of Courant algebroids in Proposition~\ref{holcourants}.
\end{proof}
\begin{remark}
The fact that the gerbe $G$ inherits a pair of holomorphic structures with respect to $X_\pm$ was proposed in the papers~\cite{MR2607445, MR2276362, MR2565045}, which contain more insights concerning the prequantization of generalized K\"ahler geometry than are formalized here. 
\end{remark}

By the holomorphic reduction procedure developed in Section~\ref{holreduct}, we saw that the Dirac structures $\overline{L}_\pm$ reduce to the pair of holomorphic Dirac structures $\Aa_\pm,\Bb_\pm\subset\EE_\pm$.  Theorem~\ref{diracflat} then immediately yields the following result, which may be interpreted as establishing a relationship between the holomorphic gerbes $\GG_\pm$ over $X_\pm$ deriving from the fact that $X_\pm$ participate in a generalized K\"ahler structure.

\begin{theorem}
Let $(\GG_\pm,\del_\pm)$ be the holomorphic gerbe with $0$-connection obtained as above from a generalized K\"ahler structure with prequantized Courant algebroid.  Then $\GG_\pm$ has flat connections over the holomorphic Lie algebroids $\Aa_\pm,\Bb_\pm$, and consequently has a flat Poisson connection with respect to the holomorphic Poisson structure $\sigma_\pm$. 
\end{theorem}
\begin{proof}
The decomposition $\EE_\pm =\Aa_\pm \oplus\Bb_\pm$ obtained in Theorem~\ref{maintransv} implies, by Theorem~\ref{diracflat}, that $\GG_\pm$ obtains flat connections over $\Aa_\pm$ and $\Bb_\pm$. By Proposition~\ref{holpoiss}, the Baer sum $\Aa_\pm^\top\boxtimes\Bb_\pm$ yields the Lie algebroid of the holomorphic Poisson structure $\sigma_\pm$, and since the Baer sum coincides with the fiber product of Lie algebroids (Proposition~\ref{baerfiber}), we obtain a flat Poisson connection on $\GG_\pm$ with respect to $\sigma_\pm$.
\end{proof}

\bibliographystyle{utphysmod}
\bibliography{GKG}
\end{document}